\documentclass[12pt]{amsart}
\usepackage{amssymb}

\usepackage{color}

\newtheorem{theorem}{Theorem}[section]
\newtheorem{lemma}[theorem]{Lemma}
\newtheorem{remark}[theorem]{Remark}
\newtheorem{definition}[theorem]{Definition}

\newtheorem{corollary}[theorem]{Corollary}
\newtheorem{proposition}[theorem]{Proposition}

\newtheorem{lem-def}[theorem]{Lemma-Definition}

\renewenvironment{proof}{{\bfseries Proof.}}{\qed}
\newcommand{\R}{\mathbb R}

\newcommand{\N}{\mathbb N}
\newcommand{\Z}{\mathbb Z}
\newcommand{\Q}{\mathbb Q}

\def\op{\operatorname}

\def\al{\alpha}

\def\alkl{\al_{k,\ell}}  
\def\ars#1{\renewcommand\arraystretch{#1}}

\def\bb{{\mathcal B}}
\def\bbl{{\mathcal B}_\ell}
\def\bct{B^{\op{ct}}}
\def\be{\beta}
\def\bed{\be_d}
\def\bel{\be_\ell}

\def\bs{\vskip.5cm}
\def\bunb{B^{\op{unb}}}

\def\chr{\op{char}}

\def\comb#1#2{\ars{0.9}\left(\!\!\begin{array}{c}
		#1\\#2
	\end{array}\!\!\right)\ars{1}}

\def\cuts{\op{Cuts}}

\def\dgn{\op{deg}_\nu}
\def\dgq{\op{deg}_Q}

\def\diso{\lower.4ex\hbox{$\downarrow$}\raise.4ex\hbox{\mbox{\scriptsize
			$\wr$}}}

\def\dta{\delta}

\def\e{\medskip}

\def\ep{\epsilon}

\def\epmin{\epsilon_{\op{min}}}
\def\epn{\epsilon_\nu}
\def\eq{\ep_Q}

\def\g{\Gamma}
\def\ga{\gamma}

\def\gaq{\ga_Q}
\def\gaqq{\ga_\qq}
\def\gar{\ga_R}
\def\gas{\ga_S}

\def\gg{\mathcal{G}}

\def\ggm{\mathcal{G}_\mu}

\def\ggn{\mathcal{G}_\nu}

\def\ggq{\mathcal{G}_Q}

\def\gmin{\ga_{\op{min}}}

\def\gn{\g_\nu}

\def\hk{\hookrightarrow}

\def\imp{\ \Longrightarrow\ }

\def\inii{\op{Init}(I)}

\def\inm{\op{in}_\mu}

\def\inn{\op{in}}

\def\inq{\op{in}_Q}

\def\inu{\op{in}_\nu}

\def\ism{\lower.3ex\hbox{\ars{.08}$\begin{array}{c}\,\to\\\mbox{\tiny $\sim\,$}\end{array}$}}
\def\iso{\ \lower.3ex\hbox{\ars{.08}$\begin{array}{c}\lra\\\mbox{\tiny $\sim\,$}\end{array}$}\ }

\def\kb{\overline{K}}
\def\kbx{\kb[x]}

\def\kh{K^h}
\def\khx{K^h[x]}

\def\kn{k_\nu}

\def\kpi{\op{KP}_\infty}

\def\kpn{\op{KP}(\nu)}

\def\kx{K[x]}

\def\la{\lambda}
\def\La{\Lambda}

\def\lcn{\op{lc}_\nu}
\def\lcq{\op{lc}_Q}

\def\lg{l\raise.6ex\hbox to.2em{\hss.\hss}l}

\def\lra{\,\longrightarrow\,}

\def\lx{\operatorname{lex}}

\def\mi{m_\infty}

\def\mlt{\op{mult}}

\def\nn{\noindent}

\def\nuq{\nu_Q}

\def\om{\omega}

\def\orb{\hbox to  .3em{$\backslash$}\backslash}
\def\ord{\op{ord}}

\def\parj{\partial_j}
\def\pbq{\overline{\pi}_Q}

\def\pib{\overline{\pi}}
\def\piq{\pi_Q}

\def\ps{\partial_s}

\def\pt{\partial_t}

\def\qq{\mathcal{Q}}

\def\rlx{\R^I_{\lx}}
\def\rlxs{\R^S_{\lx}}

\def\sg{\sigma}

\def\si{s_\infty}
\def\sii{\ \Longleftrightarrow\ }

\def\snu{\sim_\nu}

\def\sp{\op{Spec}}

\def\sub{\subseteq}

\def\supp{\op{supp}}

\def\taq{\tau_Q}

\def\ttt{\mathcal{T}}
\def\ty{\mathbf{t}}

\def\val{\op{val}}

\def\vh{v^h}

\DeclareMathOperator{\inv}{in}
\newcounter{cs}
\stepcounter{cs}
\newcommand{\casos}{\begin{itemize}}
\newcommand{\fcasos}{\end{itemize}\setcounter{cs}{1}}

\newfont{\tit}{cmr12 scaled \magstep3}

\title{Minimal limit key polynomials}
\author{Enric Nart}
\address{Departament de Matem\`{a}tiques,         Universitat Aut\`{o}noma de Barcelona,         Edifici C, E-08193 Bellaterra, Barcelona, Catalonia}
\email{enric.nart@uab.cat}
\author{Josnei Novacoski}
\address{Departamento de Matem\'{a}tica,         Universidade Federal de S\~ao Carlos, Rod. Washington Luís, 235, 13565--905, S\~ao Carlos -SP, Brazil}
\email{josnei@ufscar.br}
\thanks{Partially supported by grant PID2020-116542GB-I00  funded by the Spanish MCIN/AEI. During the realization of this project the second author was supported by a grant from Funda\c{c}\~ao de Amparo \`a Pesquisa do Estado de S\~ao Paulo (process numbers 2017/17835-9) and a grant from Conselho Nacional de Desenvolvimento Cient\'ifico e Tecnol\'ogico (process number 303215/2022-4).}

\begin{document}

\begin{abstract}
In this paper, we extend the theory of minimal limit key polynomials of valuations on the polynomial 
ring $\kx$. Minimal key polynomials are useful to describe, for instance, the defect of an extension of valued fields. We use the theory of cuts on ordered abelian groups to show that the previous results on bounded sets of key polynomials of  rank one valuations, extend to vertically bounded sets of key polynomials of valuations of an arbitrary rank. We discuss as well properties of minimal limit key polynomials in the vertically unbounded case.
\end{abstract}

\maketitle

\section*{Introduction}
Let $(K,v)$ be a valued field and let $\g$ be  the divisible hull of the value group $vK$. 

The extensions of $v$ to the polynomial ring $\kx$ were initially described by S. Mac Lane in the discrete rank one case (\cite{mcla}). This theory was developed in full generality by M. Vaqui\'e (\cite{Vaq,Vaq2}).  To every valuation $\nu$ on $\kx$, extending $v$, one can associate a countable sequence of \text{augmentations} of valuations on $\kx$:
\begin{equation}\label{chain}
\nu_0\,<\,\nu_1\,<\,\cdots\,<\,\nu_n\,<\,\cdots\,<\,\nu
\end{equation}
where we can read several numerical invariants of $\nu$. Each augmentation step requires either the intervention of an adequate \emph{key polynomial}, or it involves a limit process requiring the presence of some \emph{limit key polynomial}.

Any extension $w$ of $v$ to a finite simple extension of $K$ can be linked to one of these valuations $\nu$ on $\kx$. Vaqui\'e showed that the \emph{defect} in $w/v$ can be described in terms of the limit augmentations in the chain (\ref{chain}) (see \cite{NN,Vaq2}).

Let $p>0$ be the \emph{residue characteristic} of $v$. In \cite{Kuhl}, F.-V. Kuhlmann showed that there are two types of Artin-Schreier defect extensions (of degree $p$), which he labelled as \emph{dependent} and \emph{independent}. In \cite{KR}, Kuhlmann and Rzepka extend this classification to the case of Kummer defect extensions.   

In the rank one case, J. Novacoski and M. Spivakovsky showed that  the character dependent/independent of the defect linked to a single limit augmentation step, can be read in certain ``minimal" limit key polynomials \cite{N2023,NS2023}. Let us briefly sketch these results.

For any given valuation $\nu$ on $\kx$ and some positive integer $m$, let $\Psi_m$ be the set of key polynomials for $\nu$ of degree $m$. On $\Psi_m$ we consider the pre-ordering determined by the action of $\nu$. Assume that $\Psi_m\ne\emptyset$ and $\Psi_m$ contains no last element. Let $\qq\sub \Psi_m$ be a well-ordered strictly increasing family of key polynomials which is cofinal in $\Psi_m$.
Then, any key polynomial $F$ for $\nu$ of minimal degree strictly larger than $m$ is a limit key polynomial with respect to $\qq$.  

For a limit key polynomial $F$ for $\qq$ and every $Q\in\qq$ consider the $Q$-expansion of $F$:
\[
F=\sum_{\ell=0}^DF_{Q,\ell}Q^\ell, \quad \deg F_{Q,\ell}<m,\quad D=\lfloor \deg F/m\rfloor.
\]
For each $Q\in\qq$ let $\nuq$ be the valuation on $\kx$ obtained from $\nu$ by \emph{truncation} at $Q$. 
Consider the minimal initial segment $\dta^L$ of $\g$ containing $\left\{\nuq(F)\mid Q\in\qq\right\}$.   Let $B=\{0,\dots,D\}\setminus J$, where $J$ is the subset of all indices $\ell$ such that 
\[
\nuq(F_{Q,\ell}Q^\ell)>\dta^L\quad\mbox{  for all }\ Q\ge Q_\ell,
\]
for some $Q_\ell\in\qq$. Finally, consider the subset $\bct\sub B$ containing all $\ell\in B$ such that the values $\left\{\nuq(F_{Q,\ell})\mid Q\in\qq\right\}$ are ultimately constant. 
In the rank one case, it was shown in \cite{N2023} and \cite{NS2023} that, if the segment $\dta^L$ is bounded in $\g$, then 
\begin{equation}\label{ppower}
\bct =B\setminus\{0\}\sub p^{\N_0}=\{p^l\mid l\in \N_0\}
\end{equation}
where $\N_0$ denotes the set of all non-negative integers. A \emph{minimal limit key polynomial} would be a limit key polynomial whose $Q$-expansion has  support in $B$ (i.e. $F_{Q,\ell}=0$ for all $\ell\in J$) for all sufficiently large $Q\in\qq$.

Also, it is proven in \cite{N2023} that a defect extension of degree $p$ is dependent if and only if the corresponding set $B$ is equal to $\{0, p\}$.

The aim of this paper is to extend these results to valuations $v$ of arbitrary rank. This requires more subtle arguments at certain steps.

If $v$ has rank one, we can regard $\g$ as a subset of $\R$. Hence, if $\dta^L$ is bounded, then it admits a supremum in $\R$. This supremum can be used to estimate the differences
\begin{equation}\label{equandiffere}
\epsilon(R)-\epsilon(Q),\qquad Q<R,
\end{equation}
as $Q$ and $R$ grow in $\qq$, where $\epsilon$ is Spivakovsky's \emph{level function} (see Section \ref{subsecEp}). This is the crucial tool to derive the statements in (\ref{ppower}).

In higher rank, the situation is way more chaotic. However, in Section \ref{secCuts} we use the theory of cuts on ordered groups to introduce certain \emph{vertically bounded} cuts having a suitable kind of supremum. In Section \ref{secVB}, we assume that the cut of $\g$ determined by $\dta^L$ is vertically bounded, and we present an alternative for the estimation of the differences in \eqref{equandiffere}, leading to the results in (\ref{ppower}). 

In Section \ref{secVU} we study the vertically unbounded case and we obtain certain partial results as well.
Sections \ref{secKP} and \ref{secLKP} contain some background material on key polynomials and limit key polynomials, respectively.

 \section{Key polynomials}\label{secKP}
Let $(K,v)$ be a valued field with residue field $k=Kv$. Let $p$ be the characteristic exponent of $v$; that is, $p=1$ if $\chr(k)=0$ and $p=\chr(k)$ otherwise. Let $\kb$ be a fixed algebraic closure of $K$.

Let $\g$ be the divisible hull of the  value group $vK$.
Consider a fixed embedding  $\g\hk\La$ into a divisible ordered abelian group $\La$. 
Let $\ttt=\ttt(v,\La)$ be the tree of all $\La$-valued valuations on the polynomial ring $\kx$,
\[
 \nu\colon \kx\lra \La\cup\{\infty\},
\]
whose restriction to $K$ is $v$.
This set $\ttt$ has a partial ordering. For $\mu,\nu\in\ttt$, we define
\[
\mu\leq \nu\ \sii\  \mu(f)\leq \nu(f)\quad \mbox{for all }f\in K[x].
\]
We say that $\ttt$ is a \textbf{tree} because the intervals $
(-\infty,\nu]=\left\{\mu\in \ttt\mid \mu\le\nu\right\}$
are totally ordered for all $\nu\in\ttt$ (see \cite{VT}).
The \textbf{support} of any $\nu\in\ttt$ is the prime ideal 
\[
\mathfrak{p}_\nu:=\supp(\nu)=\nu^{-1}(\infty)\in\sp(\kx).
\]
Only the valuations with trivial support extend to a valuation on the field $K(x)$.

The value group of $\nu$ is the subgroup $\gn$ generated by $\nu\left(\kx\setminus\mathfrak{p}_\nu\right)$. The residue field $\kn$ is defined as the residue field of the valuation naturally induced by $\nu$ on the field  of fractions of  $\kx/\mathfrak{p}_\nu$.

\subsection{Graded algebra of a valuation on a ring}\label{subsecGr}

The {\bf graded algebra} of $\nu\in\ttt$ is the integral domain $\ggn=\bigoplus_{\alpha \in \gn}\mathcal P_\alpha/\mathcal P_\alpha^+$, where
\[
\mathcal P_\alpha=\{f\in K[x]\mid \nu(f)\geq \al\}\supseteq\mathcal P_\alpha^+=\{f\in K[x]\mid \nu(f)> \al\}.
\]
We say that $\al$ is the \textbf{grade} of the elements in $\mathcal P_\alpha/\mathcal P_\alpha^+$.
For $f\in K[x]\setminus \supp(\nu)$, we denote its image in $\mathcal P_{\nu(f)}/\mathcal P_{\nu(f)}^+\sub \ggn$ by $\inu f$. 

For $f,g\in K[x]\setminus \supp(\nu)$ we say that $f$ and $g$ are \textbf{$\nu$-equivalent} if $\inu f=\inu g$; that is, $\nu(f-g)>\nu(f)$.
In this case, we write $f\snu g$.\e

\noindent{\bf Definition. }{\it
	A monic $\phi\in K[x]$ is a \textbf{Mac Lane--Vaqui\'e (MLV) key polynomial}  for $\nu$ if $\inu\phi$ is a prime element in $\ggn$ and $\phi$ is $\nu$-minimal. The latter condition means that  $\deg \phi\leq \deg f$ whenever  $\inu\phi\mid \inu f$ in $\ggn$.}\e

We denote by $\kpn$ the set of all MLV key polynomials for $\nu$. They are necessarily irreducible in $\kx$.
The $\nu$-equivalence relation restricts to an equivalence relation on the set $\kpn$. For all $\phi\in\kpn$ we denote its class by
\[
[\phi]_\nu=\{\varphi\in\kpn\mid \inm\phi=\inm\varphi\}.
\]
All polynomials in $[\phi]_\nu$ have the same degree \cite[Proposition 6.6]{KP}. %We denote by $\deg[\phi]_\mu$ this common degree. 

%It is easy to characterize the existence of key polynomials.

%\begin{theorem}\cite[Theorem 4.4]{KP}\label{empty}	The set $\kpn$ is nonempty if and only if $\nu$  is valuation-transcendental. \end{theorem}

If $\mu<\nu$, then we have a homomorphism of graded algebras $\ggm\to\ggn$, defined by \[\inv_\mu(f)\longmapsto \begin{cases}\inu(f),&\mbox{ if }\mu(f)=\nu(f),\\0, & \mbox{ if }\mu(f)<\nu(f).\end{cases}\]

\noindent{\bf Definition. }{\it
	If $\mu< \nu$ in $\ttt$, then the \textbf{tangent direction} of $\mu$ determined by $\nu$ is the set  $\ty(\mu,\nu)$ of monic polynomials $\phi\in K[x]$ of minimal degree satisfying $\mu(\phi)<\nu(\phi)$.}\e

The following result follows from \cite[Theorem 1.15]{Vaq} and \cite[Corollary 2.5]{MLV}.

\begin{lemma}\label{tdef}
	Every $\phi\in \ty(\mu,\nu)$ is a MLV key polynomial for $\mu$ and $\ty(\mu,\nu)=[\phi]_\mu$.
	For all nonzero $f\in\kx$, we have $\mu(f)<\nu(f)$ if and only if $\inm \phi\mid \inm f$ in $\ggm$.
	
	Moreover, if $\mu<\nu<\eta$ in $\ttt$, then $\,\ty(\mu,\eta)=\ty(\mu,\nu)$.
\end{lemma}

For every $m\in\N$, let $\kx_m$ be the set of polynomials of degree less than $m$.

Let $\phi\in\kpn$. For any $f\in K[x]$ consider the \textbf{$\phi$-expansion of $f$}:
\[%\begin{equation}\label{phiexpansion}
f=f_0+f_1\phi+\ldots+f_n\phi^n,\qquad f_0,\dots,f_n\in\kx_{\deg(\phi)}.
\]%\end{equation}

Since $\phi$ is $\nu$-minimal, \cite[Proposition 1.3]{KP} shows that 
\begin{equation}\label{minimal}
	\nu(f)=\min\left\{\nu(f_i\phi^i)\mid 0\leq i\leq n\right\}.
\end{equation}
Consider the set $S_{\nu,\phi}(f)=\{0\le i\le n\mid \nu(f)=\nu(f_i\phi^i)\}$. Clearly,
\begin{equation}\label{inuf}
	\inu f=\sum\nolimits_{i\in S_{\nu,\phi}(f)}(\inu f_i)\pi^i,\quad \pi=\inu\phi.
\end{equation}

Let $\ggn^0\sub\ggn$ be the subalgebra generated by all homogeneous units.
If $\phi\in \kpn$ has minimal degree, then all these coefficients $\inu f_i$ are homogeneous units \cite[Proposition 3.5]{KP}; thus,  $\inu f\in\ggn^0[\pi]$.
The expression of $\inu f$ as a polynomial in $\ggn^0[\pi]$ is unique, as the following result shows (cf. \cite[Remark 16]{Dec} or \cite[Proposition 4.5]{N2021}). 

\begin{theorem}\label{g0gm}
	Let $\phi$ be a MLV key polynomial of minimal degree for $\nu$. 	Then, the prime $\pi=\inu\phi$ is transcendental over $\ggn^0$ and
	$\ggn=\ggn^0[\pi]$.
\end{theorem}

\begin{lemma}\label{deglc}
	The degree and leading coefficient  of $\inu f$ as a polynomial in $\inu \phi$ with coefficients in $\ggn^0$	 are independent of the choice of $\phi$ among all MLV key poynomials of minimal degree for $\nu$.
\end{lemma}

\begin{proof}
	Take $\phi\in\kpn$ of minimal degree in this set. By \cite[Proposition 6.3]{KP}, the other polynomials of minimal degree in $\kpn$ are obtained as
	\[
	\varphi=\phi+a,\quad a\in\kx_{\deg(\phi)},\  \nu(a)\ge \nu(\phi).
	\]
	If $\nu(a)>\nu(\phi)$, then $\inu\varphi=\inu\phi$. Otherwise, $\inu\varphi=\inu\phi+ \inu a$, with $\inu a\in\ggn^0$. The result follows immediately. 
\end{proof}\e

\noindent{\bf Definition. }{\it
	These objects are denoted by $\dgn f$, $\lcn f$ and they are called the $\nu$-\textbf{degree}  and  $\nu$-\textbf{leading coefficient} of $f$, respectively.}\e

If $d=\max(S_{\nu,\phi}(f))$, then  $\dgn f=d$ and  $\lcn(f)=\inu f_d\in\ggn^0$. Also, let us emphasize the following consequence of Theorem \ref{g0gm}.

\begin{remark}\label{units}
 A nonzero  $\inu f\in\ggn$ is a homogeneous unit if and only if $\dgn f=0$. 
 \end{remark}

\subsection{Spivakovsky's level function}\label{subsecEp}

For all $s\in \N_0$, the $s$-th Hasse-Schmidt derivative $\partial_s$ on $\kx$ is defined so that the following Taylor formula holds for any indeterminate $y$:
$$%\begin{equation}\label{taylor}
f(x+y)=\sum\nolimits_{0\le s}(\ps f)y^s \quad\mbox{for all } \,f\in\kx.
$$%\end{equation}

Take $\nu\in\ttt$ and a nonconstant $f\in\kx\setminus\supp(\nu)$. The \textbf{level} of $f$ is defined as:
$$
\epn(f)=\max\left\{\dfrac{\nu(f)-\nu(\ps{f})}s\ \Big|\ s\in\N_0\right\}\in\La.
$$
Let  $I(f)=I_\nu(f)$ be the set of all $s\in\N$ for which this maximal value is attained. 
All these elements are powers of the characteristic exponent $p$ of $v$ \cite[Proposition 2.4]{NS2018}.

We may extend this computation of $\epn$ to a function 
$$%\begin{equation}\label{ep}
\epn\colon\kx\lra \{-\infty\}\cup\La\cup\{\infty\},
$$%\end{equation}
by agreeing that 
\[
\epn(a)=-\infty,\quad \forall a\in K, \quad\qquad \epn(f)=\infty, \ \mbox{ if }\ f\ne0 \ \mbox{ and }\ \nu(f)=\infty. 
\]

%We shall make an extensive use of the following theorem by Novacoski.

\begin{theorem}\label{nov}\cite[Proposition 3.1]{N2019}
Let $\om$ be an arbitrary extension  to $\kbx$ of some $\nu\in\ttt$. Then, for every nonconstant monic polynomial $f\in \kx$ we have 
$$
\epn(f)=\max\{\om(x-a)\mid a\in \op{Z}(f)\},
$$
where $\op{Z}(f)$ is the set of roots of $f$ in $\kb$. 
\end{theorem}

As an immediate consequence of this result, we have
\begin{equation}\label{efg}
\epn(fg)=\max\{\epn(f),\epn(g)\}\quad\mbox{for all }f,g\in\kx\setminus K.	
\end{equation}
	
%Consider the image set $E_\nu=\{\epn(f)\mid f\in \kx\}$ of the function $\epn$. \e

\noindent{\bf Definition. }{\it 
Let $\nu\in \ttt$. A monic $Q\in\kx$ is said to be a \textbf{key polynomial} for $\nu$ if for every $f\in\kx$ we have}
\[
 \deg f<\deg Q\imp \epn(f)<\epn(Q).
\]

For every polynomial $Q\in\kx$, the truncated function $\nu_Q$  is a valuation  on $\kx$ such that $\nu_Q\le\nu$ (see \cite{hos,NS2018}). Recall that $\nu_Q$ is defined as follows  on $Q$-expansions:
\[
f=\sum\nolimits_{i\ge0}f_i\, Q^i,\quad  f_i\in\kx_{\deg(Q)}\ \imp\ \nu_Q(f)=\min\{\nu(f_i\,  Q^i)\mid i\ge0\}.
\]

%We say that $Q$ is a \textbf{maximal key polynomial} for $\nu$ if $\nu_Q=\nu$. Maximal key polynomials are characterized as follows.

Let $\Psi$ be the set of all key polynomials for $\nu$. For $P,Q\in\Psi$, with same degree, we have:
\begin{equation}\label{ordering}
 \nu_P<\nu_Q\sii \epn(P)<\epn(Q)\sii \nu(P)<\nu(Q). 
\end{equation}
These two equivalences were proved in \cite[Proposition 30]{Dec} and \cite[Proposition 2.10]{NS2018}, respectively.
We consider a pre-ordering on the set $\Psi$ by considering that $P\le Q$ if either $\nu(P)=\nu(Q)$, or the inequalities in (\ref{ordering}) are satisfied. In the latter case, we write $P<Q$.

\section{Limit key polynomials}\label{secLKP}
 Let $\nu\in\ttt$ and consider a positive integer $m$ such that the subset $\Psi_m\sub\Psi$ of all key polynomials for $\nu$ of degree $m$ is neither empty nor contains a maximal element.

Throughout the paper, we fix a well-ordered, strictly increasing family of key polynomials  $\qq\sub\Psi_m$, which is cofinal in $\Psi_m$. 

Let us fix the following notation for all $P,Q\in\qq$.
\[
\ars{1.4}
\begin{array}{llll}
\gaq=\nu(Q)&\qquad\eq=\epn(Q)&\qquad a_{QP}=Q-P&\qquad \g_Q=\g_{\nuq}\\
\gg_Q=\gg_{\nuq}&\qquad
\inn_Q=\op{in}_{\nuq}&\qquad 
\dgq=\deg_{\nuq}&\qquad \lcq=\op{lc}_{\nuq}.
\end{array}
\]

Let $Q_{\op{min}}=\min(\qq)$ and $\epmin=\epn(Q_{\op{min}})$. By \cite[Lemma 4.11]{VT}, we may assume that 
\[
\g_{Q_{\op{min}}}=\g_Q^0=\g_Q,\quad\mbox{ for all }Q\in\qq,
\]
where $\g_Q^0:=\left\{\nu(a)\mid a\in\kx,\ 0\le \deg(a)<m\right\}$.
In particular, all $\gaq$ belong to the divisible hull $\g$ of $vK$, and  all valuations $\nuq$ are \emph{residue-transcendental}, regardless of the properties of $\nu$.

\begin{lemma}\label{maximal}
If $P<Q$ in $\qq$, then $P$ and $Q$ are MLV key polynomials for $\nu_P$ of minimal degree. Moreover, $\ty(\nu_P,\nuq)=[Q]_{\nu_P}$.
\end{lemma}

\begin{proof}
In \cite[Corollary 2.22]{AFFGNR} it is shown that $P$ is a MLV key polynomial for $\nu_P$ of minimal degree.  
 Since $Q=P-a_{PQ}$ and $\nu(P)<\nu(Q)$, we deduce that $\nu(a_{PQ})=\nu(P)$. Hence, $Q\in\ty(\nu_P,\nuq)$, because $\nu_P(a)=\nu(a)=\nuq(a)$ for all $a\in\kx_m$ and
\[
\nu_P(Q)=\min\{\nu_P(P),\nu_P(a_{PQ})\}=\nu_P(P)=\nu(P)<\nu(Q)=\nuq(Q).
\] 
By Lemma \ref{tdef}, $Q$ is a MLV key polynomial of $\nu_P$ and $\ty(\nu_P,\nuq)=[Q]_{\nu_P}$. 
\end{proof}\e

Take $P<Q<R$ in $\qq$. By Lemma \ref{tdef}, 
\[
[R]_{\nu_P}=\ty(\nu_P,\nu_R)=\ty(\nu_P,\nuq)=[Q]_{\nu_P}.
\] 
Thus, $\pib_P:=\inn_P Q=\inn_P R$ is independent of the choice of $Q>P$ in $\qq$. In particular,  $u_P:=\inn_P a_{PQ}=\inn_P a_{PR}$, which is a unit in $\gg_P$ by Remark \ref{units}, is independent of this choice too. If  $\pi_P:=\inn_P P$, then we obtain  three homogeneous elements  of grade $\ga_P$ in $\gg_P$ related by:
\[
\pib_P=\pi_P-u_P.
\]
Under the canonical homomorphism $\gg_P\to\ggq$, the prime element $\pib_P$ is mapped to zero and $\pi_P$, $u_P$ are mapped to units. If we denote these images by the same symbols, then we have $\pi_P=u_P$ in $\ggq$.

\subsection{Stability properties of $\qq$}

\begin{lemma}\label{first}
	For all $Q<R$ in $\qq$ and all $s\in I(Q)$, we have
	\[
	\nu(\ps a_{QR})>\nu(\ps Q)=\nu(\ps R).
	\]
\end{lemma}

\begin{proof}
Write $h=a_{QR}$.
Since $\nu(h)=\gaq$, for all $s\in I(Q)$ we have
\[
\dfrac{\gaq-\nu(\ps h)}s\le\epn(h)<\eq=\dfrac{\gaq-\nu(\ps Q)}s.
\]
Hence, $\nu(\ps h)>\nu(\ps Q)$. Since $\ps R=\ps Q-\ps h$, the equality $\nu(\ps Q)=\nu(\ps R)$ follows too.
\end{proof}

\begin{lemma}\label{sstable}
	For all $Q<R$ in $\qq$, we have $\min(I(Q)) \ge \max(I(R)) $.
\end{lemma}

\begin{proof}
Write $h=a_{QR}$.
Take integers $s<t$ with $s\in I(Q)$. 	
Let us show that $t\not\in I(R)$.
Indeed, consider the two inequalities:
\[
\dfrac{\ga_R-\nu(\pt h)}t=\dfrac{\ga_R-\gaq}t+\dfrac{\gaq-\nu(\pt h)}t\le\dfrac{\ga_R-\gaq}t+\epn(h)<\dfrac{\ga_R-\gaq}s+\eq,
\]
\[
\dfrac{\ga_R-\nu(\pt Q)}t=\dfrac{\ga_R-\gaq}t+\dfrac{\gaq-\nu(\pt Q)}t\le\dfrac{\ga_R-\gaq}t+\eq<\dfrac{\ga_R-\gaq}s+\eq,
\]
Since $\pt R=\pt Q-\pt h$, we deduce from both inequalities the following one:
\[
\dfrac{\ga_R-\nu(\pt R)}t<\dfrac{\ga_R-\gaq}s+\eq=\dfrac{\ga_R-\gaq}s+\dfrac{\gaq-\nu(\ps Q)}s=\dfrac{\ga_R-\nu(\ps R)}s\le \epn(R),
\]
where we used $\nu(\ps Q)=\nu(\ps R)$ by Lemma \ref{first}. This implies that $t\not\in I(R)$.
\end{proof}\e

\begin{corollary}\label{corsstable}
There exists $Q_0\in\qq$ such that $I(Q)=I(Q_0)=\{s\}$ for all $Q\ge Q_0$. Moreover, if we denote $\taq=\gaq-s\eq$, then we have $\taq=\tau_{Q_0}$ for all $Q\ge Q_0$.
\end{corollary}

\begin{proof}
The stability of $I(Q)$ follows immediately from Lemma \ref{sstable}. Moreover, for the stable value $s$, we have $\taq=\gaq-s\eq= \nu(\ps Q)$ for all $Q\ge Q_0$. The stability of $\taq$ follows from Lemma \ref{first}. 
\end{proof}\e

\subsection{Stability properties of limit key polynomials}
For an arbitrary $f\in\kx$, let us denote the canonical $Q$-expansion of $f$ by:
\[%\begin{equation}\label{Qexp}
f=f_{Q,0}+f_{Q,1}Q+\cdots+f_{Q,n}Q^n,\qquad n=\lfloor \deg f/\deg Q\rfloor.
\]%\end{equation}

Take $R>Q$ in $\qq$. Since $Q$ and $R$ are both  MLV key polynomials for $\nuq$ (Lemma \ref{maximal}), we can compute $\nuq(f)$ as indicated in (\ref{minimal}): 
\[
\min\{\nuq(f_{Q,i}Q^i)\mid 0\le i\le n\}=\nuq(f)=\min\{\nuq(f_{R,i}R^i)\mid 0\le i\le n\}.
\]

Let $S_Q(f)=S_{\nuq,Q}(f)$ and $S_{Q,R}(f)=S_{\nuq,R}(f)$. By applying (\ref{inuf}), we get:
\begin{equation}\label{pipi}
\sum_{i\in S_Q(f)}(\inq f_{Q,i})\,\piq^i=\inq f=\sum_{j\in S_{Q,R}(f)}(\inq f_{R,j})\,\pbq^j.
\end{equation}

\begin{proposition}\label{ellstable}
For $Q\in\qq$, let $d=\dgq f$. For all  $R>Q$ in $\qq$ we have:
\begin{enumerate}
	\item[(i)] \ $\lcq(f)=\inq f_{Q,d}=\inq f_{R,d}$. In particular, $\nu(f_{Q,d})=\nu(f_{R,d})$.
	\item[(ii)] \ $d=\max(S_{Q,R}(f))\ge \deg_R f$.
\end{enumerate}
\end{proposition}

\begin{proof}
The equality $d=\max(S_{Q,R}(f))$ and statement (i) follow from Lemma \ref{deglc}, because the degree and leading coefficient of the two expressions of $\inq f$ in (\ref{pipi}) as a polynomial in $\ggq^0[\piq]=\ggq^0[\pbq]$ must coincide. 

It remains only to show that $d\ge d':=\deg_R f$. Let $S_R=S_{\nu_R,R}(f)$. The simple fact that $d\in S_{Q,R}$ and $d'\in S_R$ implies:
\[
\nuq(f_{R,d}R^d)\le \nuq(f_{R,{d'}}R^{d'}),\qquad  
\nu(f_{R,d}R^d)\ge \nu(f_{R,{d'}}R^{d'}).
\] 
From these inequalities, we immediately deduce
\[
(d'-d)\nuq(R)\ge \nu(f_{R,d})- \nu(f_{R,d'})\ge (d'-d)\nu(R).
\]
Since $\nuq(R)<\nu(R)$, this inequality implies that $d'-d\le 0$.
\end{proof}\e

The following result is an immediate consequence of Proposition \ref{ellstable}.

\begin{corollary}\label{corellstable}
There exists $Q_0\in\qq$ such that $\dgq f=\deg_{Q_0} f$ for all $Q\ge Q_0$.

Moreover, for $d=\deg_{Q_0} f$ we have $\nu(f_{Q,d})=\nu(f_{Q_0,d})$ for all $Q\ge Q_0$.
\end{corollary}

We shall denote these stable objects by $d(f)=\deg_{Q_0} f$, $\bed(f)=\nu(f_{Q_0,d})$.

We say that $f$ is $\qq$-\textbf{stable} if there exists $Q_0\in\qq$ such that $\nuq(f)=\nu_{Q_0}(f)$ for all $Q\ge Q_0$. In this case, it is obvious that $d(f)=0$.

\begin{proposition}\label{Sq}
	Suppose that $f\in\kx$ is $\qq$-unstable and let $d=d(f)$. Then, for all $Q\in\qq$ such that $\dgq f=d$ we have 
	\[
	\inq f=\lcq(f)\,\pbq^{d}\quad\mbox{and}\quad S_Q(f)=\{0,d\},
	\] 
\end{proposition}

\begin{proof}
	Take any $R\in\qq$ such that $R>Q$. The first statement follows immediately from the consideration of the Newton polygon of $f$ with respect to the pair $(\nu, R)$. Indeed, in the space $\Q\times \La$, the cloud of points $(i,\nu(f_{R,i}))$ (painted in black) has the following shape:
	
	%\begin{figure}%[h]	\caption{Newton polygon $N=\np(g)$ of $g\in \kx$. }\label{figNmodel}
	\begin{center}
		\setlength{\unitlength}{3mm}
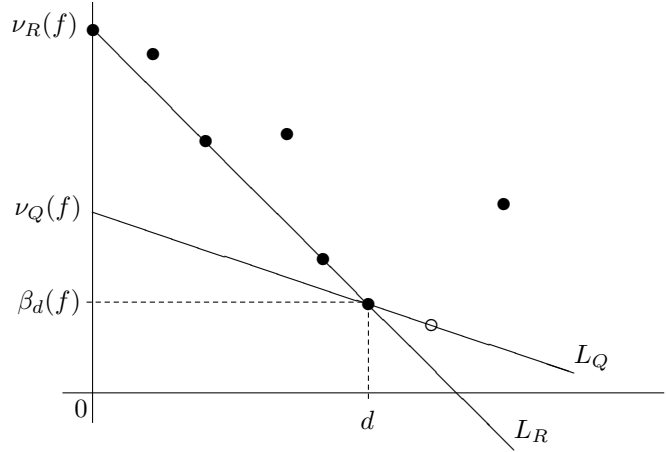
		\begin{picture}(20,14)
			\put(13.4,6){$\bullet$}\put(11,1.9){$\circ$}\put(8.8,2.65){$\bullet$}\put(3.45,8){$\bullet$}\put(7.3,4.2){$\bullet$}
			\put(6.2,8.35){$\bullet$}\put(1.75,11){$\bullet$}\put(-0.3,11.7){$\bullet$}
			\put(-1,0){\line(1,0){20}}\put(0,-1){\line(0,1){14}}
			\put(0,12.1){\line(1,-1){14}}\put(0,6){\line(3,-1){16}}
			\multiput(9.15,-0.2)(0,.3){11}{\vrule height2pt}
			%\multiput(8,.9)(0,.25){9}{\vrule height2pt}
			\put(-3.4,11.8){\begin{footnotesize}$\nu_R(f)$\end{footnotesize}}
			\put(-3.4,5.9){\begin{footnotesize}$\nuq(f)$\end{footnotesize}}
			\put(8.9,-1.2){\begin{footnotesize}$d$\end{footnotesize}}
			\put(-3.2,2.8){\begin{footnotesize}$ \bed(f)$\end{footnotesize}}
			\put(14,-1.5){\begin{footnotesize}$L_R$\end{footnotesize}}
			\put(16,1){\begin{footnotesize}$L_Q$\end{footnotesize}}
			\put(-.8,-1){\begin{footnotesize}$0$\end{footnotesize}}
			\multiput(-.2,3)(.3,0){31}{\hbox to 2pt{\hrulefill }}
			%\put(12.7,0.45){\begin{footnotesize}$\mathfrak{n}$\end{footnotesize}}
			%\put(11.6,2.5){\begin{footnotesize}$\mathfrak{n}$\end{footnotesize}}
		\end{picture}
	\end{center}
	%\end{figure}
	\bs\bs
	
	The lines $L_Q$, $L_R$ have slope $-\gaq$, $-\ga_R$, respectively. Since,
	\[
	\nuq(f)=\min_{0\le i\le n}\{\nu(f_{R,i})+i\gaq\},\qquad
	\nu_R(f)=\min_{0\le i\le n}\{\nu(f_{R,i})+i\ga_R\},
	\]
	the abscissas of the points lying on the lines $L_Q$, $L_R$ belong to the sets $S_{Q,R}(f)$, $S_R(f)$, respectively, and all other points in the cloud lie strictly above both lines.
	
	By Proposition \ref{ellstable}, $d=\max(S_{Q,R}(f))$, so that there cannot be points on the line $L_Q$ with abscissa larger than $d$ (like the point marked as a $\circ$). Hence, the point $(d,\nu(f_{R,d}))$ is the only point lying on $L_Q$. By (\ref{pipi}) and Proposition \ref{ellstable}, 
	\[
	\inq f=\inq f_{R,d}\,\pbq^d=\lcq(f)\,\pbq^d.
	\]
	
	Finally, since $\pbq=\piq-u_Q$ and $d$ is a power of the characteristic of $\ggq$, we have 
	\[
	\inq f=\lcq(f)(\piq-u_Q)^d=\lcq(f)(\piq^d-u_Q^d),
	\]
	so that $S_Q(f)=\{0,d\}$ by (\ref{pipi}).
\end{proof}\e

\begin{definition}\label{definofdefect}
If $F\in\kx$ is  monic and $\qq$-unstable of minimal degreee, then we say that $F$ is a \textbf{limit key polynomial} for $\qq$ and $d(F)$ is the \textbf{defect} of $F$ with respect to $\qq$.
Denote by $\kpi(\qq)$ the set of all limit key polynomials for $\qq$.
\end{definition}

\begin{lemma}\label{def=p}
If $F\in \kpi(\qq)$, then $d(F)$ is a power of the characteristic exponent of $v$. 
\end{lemma}

\begin{proof}
Consider the limit augmentation $\eta=[\qq;F,\nu(F)]$; that is, $\eta$ is the truncation $\eta=\nu_F$ of $\nu$ at $F$. For an arbitrary $Q\in\qq$, the defect of the augmentation $\nuq\to \eta$ is defined in \cite{NN} as $d(\nuq\to\eta):=d(F)$.  

Let $(\kh,\vh)$ be a henselization of $(K,v)$ determined by the choice of an extension of $v$ to $\kb$. It is shown in \cite{Rig} that $\vh$ and $\nuq<\eta$ admit unique common extensions $\nuq^h<\eta^h$ to $\khx$. By \cite[Theorem 1.1]{NN}, $d(\nuq\to\eta)=d(\nuq^h\to\eta^h)$. Finally,  in the Henselian case, the defect of any augmentation is a power of the characteristic exponent of $v$ by Ostrowski's Lemma \cite{Vaq2}.  
\end{proof}\e

Any two limit key polynomials $F,G\in\kpi(\qq)$ satisfy $\inq F=\inq G$ for all sufficiently large $Q\in\qq$ \cite[Lemma 4.7]{VT}. In particular, $d(F)=d(G)$, so that the defect is an intrinsic quality of $\qq$. We may denote it by $d(\qq)$. From now on, we shall simply write $d=d(\qq)=d(F)$.

Denote by $\mi$ the common degree of all the polynomials in $\kpi(\qq)$.
We are interested in finding conditions ensuring that $\mi=d  m$. 

\begin{theorem}\label{hensel=}
	If $(K,v)$ is Henselian, then $\deg F=\dgq F\deg Q$, for all $F\in\kpi(\qq)$ and all $Q\in\qq$.
\end{theorem}

\begin{proof}
	Take $F\in\kpi(\qq)$ and $Q\in\qq$. If $\inq F$ were a unit in $\ggq$, then this would imply $\nuq(F)=\nu(F)$ by Lemma \ref{tdef}. Since $F$ is $\qq$-unstable, this cannot happen. 
	
	Since $\inq F$ is not a unit in $\ggq$, then $F$ is $\nuq$-minimal by \cite[Corollary 4.5]{NN}. By \cite[Proposition 3.7]{KP},  $\deg F=\dgq F\deg Q$. 
\end{proof}\e

In particular, in the Henselian case, the relative degree $d=\dgq F$ is not only stable, but constant.
Theorem \ref{hensel=} was proved by Vaqui\'e in \cite{Vaq2} under the assumption that $Q$ was sufficiently large.
In the non-Henselian case, there is still a condition ensuring the equality  $\mi=d m$ (cf. Theorem \ref{VB}).

\subsection{Minimal limit key polynomials}
Let $D=\lfloor \mi/m\rfloor$.
For some $F\in\kpi(\qq)$ consider  the $Q$-expansions 
\[
F=F_{Q,D}Q^D+\cdots+F_{Q,1}Q+F_{Q,0}, \quad Q\in\qq.
\]
For any subset $N$ of $\{0,\ldots, D\}$, set $
F_{Q,N}:=\sum_{\ell\in N} F_{Q,\ell}Q^\ell$.

\begin{definition}\label{definofMLKP}
We say that $F$ is a \textbf{minimal limit key polynomial for $\qq$} if  $F=F_{Q,N}$ for some $Q\in \qq$, $N\sub\{0,\dots,D\}$, and the following two conditions hold:
\begin{itemize}
	\item \ $F_{R,N}\in\kpi(\qq)$ for all $R\ge Q$ in $\qq$.
	\item \ For every $\ell\in N$ there exists a cofinal subset $\mathcal{S}_\ell$ of $\qq$ such that $F_{Q,N\setminus\{\ell\}}\not\in\kpi(\qq)$ for every $Q\in \mathcal{S}_\ell$. 
\end{itemize}
\end{definition}

The idea behind this definition is the following. 
%We shall see in Section \ref{secVB} that this set $J$ and the indices $i_\ell$ do not depend on the choice of the limit key polynomial $F$.\e
Let $\dta^L$ be the smallest initial segment of  $\g$ containing $\{\nuq(F)\}_{Q\in \qq}$. Suppose that $F_{Q,k}=0$ for some $Q\in\qq$ and some index $0<k<D$. If we take a non-zero $a\in K[x]$, with $\deg(a)<m$, such that $
\nuq(aQ^k)>\delta^L$, then $G:=F+aQ^k$
is a limit key polynomial for $\qq$ too (cf. Lemma \ref{MinAug} below). However, the term $G_{Q,k}Q^k=aQ^k$ can be seen as a superfluous monomial of $G$. Thus, a minimal limit key polynomial is a limit key polynomial having no superfluous terms in a cofinal family of $Q$-expansions.

If $\qq$ admits a limit key polynomial $F$, then it also admits minimal limit key polynomials. Indeed, consider the sets  $J$ and $B$ defined in the Introduction.

\begin{proposition}\label{existence}
Take $Q\ge Q_\ell$ for all $\ell\in J$.
	Then, $F_{Q,B}$
	is a minimal limit key polynomial for $\qq$.
\end{proposition}

\begin{proof}
	For every $R\ge Q$ in $\qq$ we have
	\[
	\nu_R\left(F-F_{Q,B}\right)=\nu_R\left(\sum_{\ell\in J}F_{Q,\ell}Q^\ell\right)\geq\min_{\ell\in J}\left\{\nu_R\left(F_{Q,\ell}Q^\ell\right)\right\}> \delta^L.
	\]
	Since $\nu_R(F)\in \delta^L$, this implies that $\nu_R(F)=\nu_R(F_{Q,
		B})$. Hence, $F_{Q,B}$ is a limit key polynomial for $\qq$. 
	Now, if we take $G:=F_{Q,B}$, this argument shows that $G_{R,B}\in\kpi(\qq)$ for all $R\ge Q$. 
	
	Finally, take $\ell\in B$. Since $\ell\notin J$, there exists a cofinal subset $\mathcal{S}_\ell$ of $\qq$ such that
	$
	\nu_R(G_{R,\ell}R^\ell)> \delta^L$ for every $R\in\mathcal{S}_\ell$.
	Set $N=B\setminus\{\ell\}$. For all $R\in \mathcal{S}_\ell$, we have 
	\[
	G_{R,B}=G_{R,N}+G_{R,\ell}R^\ell.
	\]
	For all $S\ge R$ in $\qq$, the value  $\nu_S\left(G_{R,\ell}R^\ell\right)=\nu_R\left(G_{R,\ell}R^\ell\right)\in \dta^L$ remains constant, while $\nu_S(G_{R,B})$ is cofinal in $\dta^L$. Thus, $G_{R,N}$ is $\qq$-stable and is not a limit key polynomial. 
\end{proof}\e

Observe that this polynomial depends on the choice of $Q$, hence it is not unique.

\section{Cuts in ordered abelian groups}\label{secCuts}
%\subsection{}
All ordered sets considered in this paper are assumed to be totally ordered.

Let $I$ be an ordered set. We let $I\infty$ be the ordered set obtained by adding a (new) maximal element, which is formally denoted as $\infty$.

For $S,T\sub I$ and $i\in I$, the following expressions
$$i<S,\ \quad i>S,\ \quad i\le S,\ \quad i\ge S,\ \quad S<T,\ \quad S\le T$$
mean that the corresponding inequality holds for all $x\in S$ and all $y\in T$.

An \textbf{initial segment} of $I$ is a subset $S\sub I$ such that $I_{\le i}\sub S$ for all $i\in S$.
%A \textbf{final segment} of $I$ is a subset $T\sub I$ such that $I_{\ge i}\sub T$ for all $i\in T$.
On the set $\inii$ of all initial segments of $I$ we consider the ordering determined by ascending inclusion. %It has a minimal and a maximal element:$$\emptyset=\min(\inii),\qquad I=\max(\inii).$$

Let  $\cuts(I)$ be the set of all \textbf{cuts} in $I$.
A cut in $I$ is a pair $\dta=(\dta^L,\dta^R)$ of subsets of $I$ such that $$\dta^L< \dta^R\quad\mbox{ and }\quad \dta^L\cup \dta^R=I.$$ Since $\dta^L$ is an initial segment of $I$, we
%and $\dta^R=I\setminus\dta^L$ is a final segment. 
have a natural bijection $\inii\to\cuts(I)$, determined by $S  \longmapsto (S,I\setminus S)$.
In particular, $\cuts(I)$ inherits a total ordering and  admits a minimal element $(\emptyset,I)$ and a maximal element $(I,\emptyset)$. These elements are said to be the \textbf{improper cuts}.
%All other cuts are \emph{proper}. The proper cuts are also called \emph{Dedekind} cuts.

%Thus, we sometimes abuse language and identify a cut with the corresponding initial segment.
For all $S\sub I$, we denote by  $S^+$, $S^-$ the cuts determined by the initial segments
\[
\{i\in I\mid \exists x\in S: i\leq x\},\qquad \{i\in I\mid i<S\},
\]
respectively.
If $S=\{i\}$, then we will write $i^+=(I_{\le i},I_{>i})$ instead of $\{i\}^+$ and $i^-=(I_{<i},I_{\ge i})$ instead of $\{i\}^-$. These cuts are said to be 
\textbf{principal}. 

%Hence, a cut $D$ is principal if  either $\dta^L$ has a maximal element, or $\dta^R$ has a minimal element. A cut for which these two conditions hold simultaneously is called a \emph{gap}, or a \emph{jump} according to different authors.

If $I\hk I'$ is an extension of ordered sets and $j \in I'$ satisfies
$\dta^L\leq j\leq \dta^R$, then we say that \textbf{$j$
	realizes} the cut $\dta$ in $I'$.

\subsection{Hahn's embedding theorem}\label{subsecHahn}
Let $\g$ be a divisible ordered abelian group.
For a subgroup $H\sub \g$, the quotient $\g/H$ inherits a structure of ordered group if and only if $H$ is \textbf{convex}; that is, 
$$
h\in H,\ h>0 \imp  \left[0,h\right]\sub H.
$$
In this case, we may define an ordering in $\g/H$ by:
$$
a+H<b+H \sii a+H\ne b+H \ \mbox{ and }\ a<b.
$$
The notation $a+H<b+H$ is compatible with its meaning as subsets of $\g$; that is, for all $c\in a+H$, $d\in b+H$, we have $c<d$.

The convex subgroups of $\g$ are totally ordered by inclusion. The order-type of the set of proper convex subgroups is the $\textbf{rank}$ of $\g$. The convex subgroups of $\g/H$ are in 1-1 correspondence with the convex subgroups of $\g$ containing $H$.

%\subsection*{Principal convex subgroups}

For any $a\in\g$, the convex subgroup of $\g$ generated by $a$ (smallest convex subgroup containing $a$) is said to be \textbf{principal}. 

Let $I$ be the set of {\bf nonzero} convex principal subgroups of $\g$, ordered by {\bf descending} inclusion.

We identify $I\infty$ with a set of indices parametrizing all  principal convex subgroups of $\g$. For all $i\in I$ we let $H_i$ be the corresponding  principal convex subgroup. Also, we agree that $H_\infty=\{0\}$. According to our convention, for all $i,j\in I\infty$, we have
$$
i<j \sii H_i\supsetneq H_j.
$$

%For all $i\in I$, the subgroup $H_i$ has an immediate predecessor in the set of all convex subgroups:\[H_i^*\subsetneq H_i.\] If $H_i$ is generated (as a convex subgroup) by $a\in\g$, then $H_i^*$ is the union of 

To each initial segment $S\sub I$,  we may associate the convex subgroup:
\[
H_S=\bigcup\nolimits_{i\in I\setminus S}H_i.
\]
The assignment $S\mapsto H_S$ parametrizes all convex subgroups of $\g$. The nonzero principal convex subgroups 
correspond to initial segments $S$ such that $I\setminus S$ has a minimal element. The trivial subgroup corresponds to $S=I$.

Denote by $\rlx$ the \textbf{Hahn product}; that is, $\rlx$ is the subgroup of the cartesian product $\R^I$ formed by the elements $x=\left(x_i\right)_{i\in I}$  whose support
\[
\supp(x)=\{i\in I\mid x_i\ne0\}\sub I
\]
is a well-ordered subset, with respect to the ordering induced by $I$. It makes sense to consider the lexicographical ordering on $\rlx$.

By Hahn's theorem, we may choose a (non-canonical) embedding $\g\hk\rlx,
$ of ordered groups such that the principal convex subgroups of $\rlx$ are parametrized by the set $I$ as follows: 
\[%\begin{equation}\label{Hi}
H_i=\left\{(x_j)\in\rlx\mid x_j=0\ \mbox{ for all }j<i\right\},\quad i\in I.
\]%\end{equation}

For any convex subgoup $H\sub \g$, let $H_\R$ be the convex subgroup of $\rlx$ generated by the image of $H$ under the embedding $\g\hk\rlx$. The assignment $H\mapsto H_\R$ is a bijection between the sets of convex subgroups of $\g$ and $\rlx$, and restricts to a bijection between the corresponding subsets of principal convex subgroups. We shall often abuse language and drop the subindex  $(\ )_\R$. For instance, we write
\[
H_S=\left\{(x_j)\in\rlx\mid x_j=0\ \mbox{ for all }j\in S\right\},
\] 
instead of $\left(H_S\right)_\R$.
Finally, each nonzero principal convex subgroup $H_i$ has a maximal proper convex subgroup:
\[
H_i^*:=\left\{(x_j)\in\rlx\mid x_j=0\ \mbox{ for all }j\le i\right\}\subsetneq H_i.
\]
This convex subgroup $H_i^*$ is not necessarily principal. Note that the quotient $H_i/H_i^*$  is isomorphic to $\R$, as an ordered group.
  
\subsection{Cuts in $\g$}
Let $\dta=(\dta^L,\dta^R)$ be a cut in $\g$. Cuts in $\g$ may be shifted by elements $a\in\g$ and multiplied by $-1$:  
\[
a+\dta=(a+\dta^L,a+\dta^R),\qquad -\dta=(-\dta^R,-\dta^L).
\]
Every cut has a (symmetric) \textbf{invariance group} defined as:
\[
H(\dta)=\left\{h\in\g\mid h+\dta^L=\dta^L\right\}=\left\{h\in\g\mid h+\dta^R=\dta^R\right\}.
\] 
This is a convex subgroup which is invariant under shifting or multiplication by $-1$:
\[
H(a+\dta)=H(\dta)=H(-\dta)\quad \mbox{ for all }a\in\g.
\]

For a positive $h\in\g$ we have 
\begin{equation}\label{practice}
h\in H(\dta)\ \sii\ h+\dta^L\sub \dta^L.
\end{equation}
Indeed, for any $x\in\dta^L$ we have $x=h+(x-h)$ and $x-h$ belongs to the initial segment $\dta^L$ because $x-h<x$. Thus, $\dta^L\sub h+\dta^L$. \e

%Let $S\sub I$ be the initial segment for which $H(\dta)=H_S$. Consider the convex subgroup $H(\dta)\sub H'(\dta)$ which is the only possible candidate to be an immediate successor of $H(\dta)$ in the set of all convex subgroups; that is,\[H(\dta)=H_S=\bctgcup\nolimits_{i\not \in S}H_i\sub H'(\dta):=\bctgcap\nolimits_{i\in S}H_i.\]

%If $S$ contains no maximal element, then $H(\dta)=H'(\dta)$ has no immediate successor. If $\max(S)=\{i_0\}$, then $H'(\dta)=H_{i_0}$ is principal and it is the immediate successor of $H(\dta)=H_{i_0}^*$ \cite[Lemma 2.2]{csme}.  In the latter case, we say that $\dta$ has a \textbf{convexity gap}.\e

\noindent{\bf Definition. }{\it A \textbf{ball cut} is either $\dta=\left(a+H\right)^+$, or $\dta=\left(a+H\right)^-$, for some $a\in\g$ and some convex subgroup $H$. Equivalently, the induced cut $\dta/H=\left(\dta^L/H,\dta^R/H\right)$ in $\g/H$ is principal.
		All other cuts are said to be  \textbf{non-ball cuts}}.\e 

Clearly, for a ball cut $\dta=\left(a+H\right)^\pm$ we have $H(\dta)=H$.%Every cut falls in one (and only one) of the six types indicated in Figure \ref{FigTypeCuts}.

%\begin{figure}[h]	\caption{The six types of cuts. }\label{FigTypeCuts}\e	\begin{center}\ars{1.2}	\begin{tabular}{|c|c|c|}			\hline			$\dta$&{\bf convexity gap}&{\bf notation} \\\hline			$\left(a+H\right)^+$& yes&$\mbox{\tt (b+G)}^+$ \\\hline			$\left(a+H\right)^+$& no&$\mbox{\tt (b+NG)}^+$ \\\hline			$\left(a+H\right)^-$& yes&$\mbox{\tt (b+G)}^-$ \\\hline			$\left(a+H\right)^-$& no&$\mbox{\tt (b+NG)}^-$ \\\hline			non-ball& yes&{\tt nb+G} \\\hline			non-ball& no&{\tt nb+NG} \\\hline		\end{tabular}	\end{center}\end{figure}

Let us exhibit examples of non-ball cuts. 
Take $\g=\Q$ and $\xi\in\R\setminus\Q$. Then, $\dta= \left(\Q_{<\xi},\Q_{>\xi}\right)$ is a non-ball cut. Similarly, let $\g$ be the subgroup of the direct product $\Q^{\N}$ consisting of all vectors with a finite number of nonzero coordinates. Then, for every $\xi\in\Q^\N\setminus\g$ the cut $\dta=\left(\g_{<\xi},\g_{>\xi}\right)$
is non-ball too.

In both cases,  $H(\dta)=0$, so that the non-ball character is determined by the fact that they are non-principal cuts.\e

%The shift $\dta\mapsto \dta+a$ preserves the type of all cuts.However, multiplication by $-1$ acts on the six types as follows$$%\begin{equation}\label{types-1}\mbox{\tt (b+G)}^+\ \leftrightarrow\ \mbox{\tt (b+G)}^-,\qquad \mbox{\tt (b+NG)}^+\ \leftrightarrow\ \mbox{\tt (b+NG)}^-,\qquad \mbox{\tt nb+G}\,\circlearrowleft,\qquad \mbox{\tt nb+NG}\,\circlearrowleft. $$%\end{equation}

%\subsection{Cuts admitting a vertical supremum}\mbox{\null}\e

\noindent{\bf Definition. }{\it Let $\dta$ be a cut in $\g$ with invariance group $H=H(\dta)$. We say that $\dta$ admits a \textbf{vertical supremum} if the set of classes
	\[
	\dta^L+H=\left\{x+H\mid x \in\dta^L\right\}\sub \g/H
	\] 
contains no maximal element.}\e

\begin{lemma}\label{VS}
For all $\dta\in\cuts(\g)$, the set $\dta^L+H$ admits a supremum as a subset of $\rlx/H$. Moreover, the only cuts which do not admit a vertical supremum are those of the form $\dta=\left(a+H\right)^+$.
\end{lemma}

\begin{proof}
If $\dta=\left(a+H\right)^+$, then $\sup\left(\dta^L+H\right)=a+H $, which belongs to $\dta^L+H$.

If $\dta=\left(a+H\right)^-$, then $\sup\left(\dta^L+H\right)=a+H $, which does not belong to $\dta^L+H$.

If $\dta$  is a non-ball cut, then $\dta$ is realized by some $\xi\in\rlx$ \cite[Section 4.1.1]{csme}. Since $\dta^L$ has no maximal element in $\g$ ($\dta\ne a^+$) and $\dta^R$ has no minimal element in $\g$ ($\dta\ne a^-$), this realization must be strict: $\dta^L<\xi<\dta^R$. Thus, for $H=H(\dta)$ we have
\[
\dta^L<\xi+H<\dta^R.
\]
Hence, $\sup\left(\dta^L+H\right)=\xi+H $, and this class does not belong to $\dta^L+H$.
\end{proof}\e

The \textbf{natural valuation} on $\rlx$ is the mapping 
$$
\val\colon \rlx\lra I\infty,\qquad  \xi=(\xi_i)_{i\in I}\longmapsto \val(\xi)=\min\{i\in I\mid \xi_i\ne0\}.  
$$
%assigning to each $a\in\g$ the principal convex subgroup generated by $a$.

This mapping satisfies the following two properties, for all $\xi,\eta\in\g$:\e
\begin{itemize}
\item $\val(\xi)=\infty \sii \xi=0$.
\item $\val(\xi-\eta)\ge\min\{\val(\xi),\val(\eta)\}$, and equality holds if $\ \val(\xi)\ne\val(\eta)$
\end{itemize}
\e
The lexicographical ordering in $\rlx$ can be characterized as follows:
\[
\xi=(\xi_j)_{j\in I}>\eta=(\eta_j)_{j\in I}\ \sii \ \xi_i>\eta_i\ \mbox{ for }i=\min\{\val(\xi),\val(\eta)\}.
\]

\noindent{\bf Definition. }{\it We say that $\dta\in\cuts(\g)$ is \textbf{vertically bounded (VB)} if for all $x\in\dta^L$, $q\in \Q_{>1}$, we have $q(y-x)>\dta^L-x$ for some $y\in\dta^L$.}
	
\begin{proposition}\label{VS=VB}
For all $\dta\in\cuts(\g)$ the following conditions are equivalent.
\begin{enumerate}
\item[(a)] \ $\dta$ admits a vertical supremum. 
\item[(b)] \ $\dta$ is vertically bounded. 
\item[(c)] \ For some $x\in\dta^L$ and all $q\in \Q_{>1}$, we have  $q(y-x)>\dta^L-x$ for some $y\in\dta^L$.
\end{enumerate}
%A cut in $\g$ is vertically bounded if and only if it admits a vertical supremum.
\end{proposition}

\begin{proof}
Let $H=H(\dta)$. Suppose that $\dta\in\cuts(\g)$ admits a vertical supremum. By Lemma \ref{VS}, there exists $\xi=(\xi_i)_{i\in I}\in \rlx$ such that $\sup(\dta^L+H)=\xi+H$ in $\rlx/H$ and, moreover,  $\xi+H>x+H$ for all $x\in\dta^L$.  In particular, $\xi>\dta^L$.

Take any $x=(x_i)_{i\in I}\in\dta^L$ and a rational number $q>1$. Let $i=\val(\xi-x)\in I$.  

Suppose that the following condition  holds:
\begin{equation}\label{cond6}
\val(\xi-x)<\val(\xi-y)\quad\mbox{ for some }y\in\dta^L,\ y>x.
\end{equation}
In this case,
\[
\val(q(y-x))=\val(y-x)=\val((\xi-x)-(\xi-y))=\val(\xi-x)=i.
\]
Since $x_i<\xi_i=y_i$, we have $q(y_i-x_i)>y_i-x_i=\xi_i-x_i$. Thus, $q(y-x)>\xi-x>\dta^L-x$.

Now, suppose that  condition (\ref{cond6}) does not hold. Then,
\[
\val(\xi-x)=\val(\xi-y)=i\quad\mbox{ for all }y\in\dta^L,\ y>x.
\]
In this case, we claim that 
\[
H=H_i^*=\left\{h=(h_j)_{j\in I}\in\rlx\mid \val(h)>i\right\}.
\]
Indeed, if $h\notin H_i^*$, $h>0$, then $j:=\val(h)\le i$ and $h_j>0$. If $j<i$, then $x+h>\xi>\dta^L$. If $j=i$, then for a sufficiently large positive integer $n$ we have $x+nh>\xi>\dta^L$  as well. Hence, $h\notin H$, so that $H\sub H_i^*$. 

On the other hand, for all $y\in\dta^L$, $y>x$, we have $y+H_i^*<\xi+H_i^*$, because $y_i<\xi_i$. Hence, $y+H_i^*<\xi+H$, so that $y+H_i^*\sub \dta^L$. As shown in (\ref{practice}) this implies $H_i^*\sub H$. This ends the proof of the equality $H=H_i^*$.
Now, we claim that 
\[
\xi_i=\sup(Y),\qquad Y:=\{y_i\mid y=(y_j)_{j\in I}\in\dta^L,\ y>x\}\subset \R. 
\] 
 Indeed, $\xi_i$ is an upper bound for $Y$. Now, suppose that $\xi_i>c$ for some $c\in\R$ which is an upper bound of $Y$ too. Then, the element $\eta=(\eta_\ell)_{\ell\in I}\in\rlx$ determined by
 \[
\eta_\ell =\xi_\ell \ \mbox{ for all }\ell<i;\quad \eta_i=c\quad \eta_\ell=0  \ \mbox{ for all }\ell>i,
 \] 
satisfies $\xi+H>\eta+H\ge y+H$ for all $y\in \dta^L$, and this contradicts the hypothesis that $\xi+H$ was the supremum of all $y+H$ for $y\in \dta^L$. This ends the proof of the equality $\xi_i=\sup(Y)$.

As a consequence, we may achieve $q(y_i-x_i)>\xi_i-x_i$ for a sufficiently large $y_i\in Y$. This clearly implies $q(y-x)>\xi-x>\dta^L-x$. This ends the proof that (a)$\Rightarrow$(b). 

The implication (b)$\Rightarrow$(c) is obvious. Thus, we need only to show that (c)$\Rightarrow$(a).

Suppose that $\dta$ does not admit a vertical supremum. By Lemma \ref{VS}, $\dta=\left(a+H\right)^+$ for some $a\in \g$ and some convex subgroup $H$. Then, $H=H(\dta)$.
Shifting by the element $a$ preserves both conditions (a) and (c). Thus, we may assume that $\dta=H^+$, so  that $\dta^L$ is the minimal initial segment of $\g$ containing $H$. 

In order to show that (c) does not hold, it suffices to check that the following inequality is never satisfied:
\[
q(y-x)>H-x,\qquad x,y\in\dta^L, \ x<y, \ q\in\Q_{>1}.
\] 
Since the left-hand side grows with $y$ and there exists $h\in H$ such that $y\le h$, we may assume that $y\in H$. If $x<H$, then 
\[
q(y-x)>H-x>H-qx\ \imp\ qy>H,
\] 
which is impossible because $qy\in H$. Finally, if $x\in H$, then $q(y-x)>H-x=H$ is impossible too, because $q(y-x)\in H$.
\end{proof}%\e

\section{Vertically bounded sets of key polynomials}\label{secVB}
Let us go back to our well-ordered, strictly increasing  family
$\qq$ of key polynomials for $\nu$ of degree $m$, admitting limit key polynomials  of degree $\mi>m$.

Let $\dta=\dta_\qq$ be the cut in $\g$ determined by taking $\dta^L$ to be the minimal initial segment of $\g$ containing the set $\{\nuq(F)\mid Q\in \qq\}$ for some $F\in \kpi(\qq)$.
By Lemma \ref{MinAug} below, this cut is independent of the choice of $F$.

Let $\ga=\gaqq$ be the cut in $\g$ determined by taking $\ga^L$ to be the minimal initial segment of $\g$ containing the set $\{\gaq\mid Q\in \qq\}$. \e

\nn{\bf Definition. }{\it We say that $\qq$ is vertically bounded if the cut $\gaqq$ is vertically bounded.}\e

If the valuation $v$ has rank one, then $\qq$ is VB if and only if the set $\{\gaq\mid Q\in\qq\}$ is bounded in $\g$.

\begin{theorem}\label{VB}\cite[Theorem 4.9]{AFFGNR}
	If $\qq$ is VB, then $\mi=d m$. %In particular, $d=D$.
\end{theorem}

\begin{lemma}\cite[Section 7.3.3]{VT}\label{MinAug} Let $F\in\kpi(\qq)$. Then,
	\[
	\kpi(\qq)=\left\{F+a\mid a\in\kx,\ \deg a<\mi,\ \nu(a)>\dta^L\right\}.
	\]
\end{lemma}

It is easy to deduce from this result that the sets $B,J\sub\{0,\dots,d\}$ defined in the Introduction do not depend on the choice of $F$ in $\kpi(\qq)$.

From now on, we assume that $\qq$ is vertically bounded. 

In this section, we show that the results of \cite{N2023,NS2023}, analyzing limit key polynomials of  rank one valuations, extend to VB  sets $\qq$ of key polynomials for a valuation $\nu$ of arbitrary rank.

\subsection{Comparison of $Q$-expansions}

For all $Q\in\qq$, consider the canonical $Q$-expansion of $F$:
\[
F=F_{Q,d}Q^d+\cdots+F_{Q,1}Q+F_{Q,0},
\]  
where $F_{Q,\ell}\in\kx_m$ for all $0\le\ell\le d$.

The ``comparison of $Q$-expansions" problem seeks for the detection of a regular behaviour for the coefficients $F_{Q,\ell}$ for a fixed index $0<\ell\le D$ and variable $Q\in\qq$.

Denote $\be_{Q,\ell}=\nu(F_{Q,\ell})$ for all $Q$ and all $\ell$. For a fixed $\ell$, we say that the values $\be_{Q,\ell}$ are \textbf{ultimately constant} if there exists $Q_0\in\qq$ large enough so that
\[
\be_{Q,\ell}=\be_{Q_0,\ell}\quad\mbox{ for all }\,Q\ge Q_0.
\] 
%In this case,  we denote this constant value by $\beta_\ell$.

We recall that $\bct\sub B$ is the set of all ultimately constant indices in the set $B$.

Let us fix some data associated to $\qq$.\e

$\bullet$ \ We denote by $H=H(\gaqq)$ the invariance group of the cut $\gaqq$ in $\g$.\e

$\bullet$ \ 
We fix $b\in\rlx$ any element in the ``vertical supremum class":
\[
 b+H=\sup\{\gaq+H\mid Q\in\qq\}\ \,\mbox{ in }\,\rlx/H,
\]
whose existence is guaranteed by Proposition \ref{VS=VB}.\e

$\bullet$ \ 
We fix $F\in \kpi(\qq)$. By Theorem \ref{VB}, $\deg(F)=dm$. 
Since $F$ is monic, this implies $F_{Q,d}=1$ for all $Q\in\qq$. Also, $\dgq F=d$  for all $Q\in\qq$, by Proposition \ref{ellstable}. 
In particular, $\nuq(F)=d\gaq$ for all $Q\in\qq$ large enough, so that $\dta_\qq=d\gaqq$.\e

We start by reformulating the condition of vertical boundedness as an inequality involving the values $\{\ep_Q\mid Q\in\qq\}$.

\begin{lemma}\label{basic0}
If $\qq$ is VB, then there exists $Q\in \qq$  such that  $\ep_Q-\epmin>d(b-\ga_Q)$. 
\end{lemma}

\begin{proof}
Take $P\in\qq$ large enough to stabilize $d$ and the set $I(P)=\{s\}$. For any $Q>P$ in $\qq$, Corollary \ref{corsstable} shows that  
\begin{equation}\label{epga}
\ep_Q-\ep_P=\left(\ga_Q-\ga_P\right)/s.
\end{equation}
Since $\ep_Q-\epmin\ge \ep_Q-\ep_P$, it suffices to show the existence of $Q>P$ such that
\[
\ga_Q-\ga_P>sd(b-\ga_Q).
\]
By writing $b-\ga_Q=\left(b-\ga_P\right)+\left(\ga_P-\ga_Q\right)$, this inequality  is equivalent to
\[
\frac{1+sd}{sd}\left(\ga_Q-\ga_P\right)>b-\ga_P.
\]
This inequality is guaranteed by  vertical boundedness because $b<\ga^R$.
\end{proof}\e

From now on, we fix $P\in\qq$ large enough so that 
\begin{equation}\label{cond1}
I(P)=\{s\}\quad \mbox{and}\quad \ep_P-\epmin>d(b-\ga_P).
\end{equation}
The last inequality is guaranteed by Lemma \ref{basic0}.
Let us quote \cite[Lemma 2.1]{NS2023}.

\begin{lemma}\label{basic00}
Let $Q$ be a key polynomial for $\nu$. Suppose that $t=r+qQ$ in $\kx$ and $\ep:=\max\{\epn(t),\epn(r)\}<\ep_Q$. Then, $\nuq(qQ)-(\ep_Q-\ep)\ge \nu(t)=\nu(r)$.
\end{lemma}

 Along this section, we shall apply Lemma \ref{basic0} to polynomials $t$, $r$ which are a product of polynomials in $\kx_m$. The equality in (\ref{efg}) shows that $\epn(t),\epn(r)<\epmin$. Hence,   
Lemma \ref{basic0} and the assumption (\ref{cond1}) imply that:
\begin{equation}\label{basic}
\nuq(qQ)-d(b-\ga_P)> \nu(t)=\nu(r)\quad\mbox{ for all }\,Q>P.
\end{equation}

For  $Q>P$, let us denote $h_Q=P-Q\in\kx_m$. Note that $\inq h_Q=\inq P$.

For any $g\in\kx$, its $P$-expansion can be reinterpreted as the evaluation at $P$ of a certain polynomial $L\in\kx_m[X]$:
\[
L=\sum\nolimits_{0\le j}g_{P,j}X^j\ \imp\ g=L(P).
\]
As a previous step for the comparison between $P$-expansions and $Q$-expansions of the same polynomial, Novacoski and Spivakovsky compared in \cite[Proposition 2.4]{NS2023} the $\nu$-values of $g=L(P)$ and the polynomial $L(h_Q)$. 
We reproduce the proof of this result, because we  need a slightly stronger statement. 

\begin{proposition}\label{NS2.4}
Let $g=L(P)\in\kx$ for some $L\in\kx_m[X]$ with $\deg_X L\le d$. Then, for all $Q>P$ we have $\nuq(L(h_Q))\ge\nuq(g)$ and equality holds if and only if $\,0\in S_Q(g)$. In this case, $\inq L(h_Q)$ is a unit in $\ggq$.
\end{proposition}

\begin{proof}
Denote $h=h_Q$. Consider the canonical $P$ and $Q$-expansions of $g$:
\[
g=a_0+a_1P+\cdots+a_nP^n=b_0+b_1Q+\cdots +b_nQ^n.
\]
Let $t=\deg_P g=\max(S_P(g))$. By Proposition \ref{ellstable}, $t=\max(S_{P,Q}(g))$ and $\inn_P a_t=\inn_P b_t$. On the other hand, the canonical $Q$-expansions of $g$ and $L(h)$ have the same $0$-th coefficient. Indeed,
\[
g=\sum_{j=0}^na_j(Q+h)^j=\sum_{j=0}^na_jh^j+QA=L(h)+QA,
\]
for some $A\in\kx$. Let us write the canonical $Q$-expansion of $L(h)$ as follows:
\[
L(h)=a_0+a_1h+\cdots+a_nh^n=b_0+c_1Q+\cdots +c_nQ^n.
\]

Now, consider the divisions with remainder:
\[
a_jh^j=a_{j,0}+q_jQ,\quad\deg a_{j,0}<m,\quad  1\le j\le n.
\]
For all $j>0$, the inequality in (\ref{basic}) shows that
\[%\begin{equation}\label{Ineq}
\begin{array}{rl}
\nuq(q_jQ)>&\!\nu(a_jh^j)+d(b-\ga_P)\ge\nu(a_th^t)+t(b-\ga_P)\\
=&\!\nu(b_t)+tb>\nu(b_t)+t\gaq=\nu(b_tQ^t)\ge \nuq(g).
\end{array}
\]%\end{equation}
%Since $b_0=a_0+a_{1,0}+\cdots +a_{n,0}$, 
These inequalities imply that $\nu(c_kQ^k)>\nuq(g)$ for all $k>0$. Since $\nu(b_0)\ge \nuq(g)$, we deduce that $\nuq(L(h_Q))\ge\nuq(g)$ and equality holds if and only if
$\nu(b_0)=\nuq(g)$.

In this case, we have $\nu(c_kQ^k)>\nu(b_0)$ for all $k>0$, so that  $\inq L(h)=\inq b_0$ is a unit in $\ggq$.
\end{proof}%\e

\subsection{Universal $Q$-expansions of limit key polynomials}
Consider the polynomial $L\in\kx_m[X]$ determined by the canonical $P$-expansion of  $ F$:
\[
L(X)=F_{P,0}+F_{P,1}X+\cdots +F_{P,d-1}X^{d-1}+X^d,\qquad L(P)=F.
\] 
We consider the Hasse-Schmidt derivatives $\parj L$ in the ring $K(x)[X]$. The Taylor formula (cf. Section \ref{subsecEp}) applied to $x=h_Q$, $y=Q$ determines a universal $Q$-expansion of $F$, for all $Q\in\qq$:
\[
F=L(h_Q+Q)=\sum_{j=0}^d\parj L(h_Q)\,Q^j.
\] 

Of course, these are non-canonical $Q$-expansions. However, for $Q>P$ sufficiently large, Proposition \ref{NS2.4} leads to a comparison with the canonical $Q$-expansion. 

Let us denote $A_j=\parj L(h_Q)$ for all $0\le j\le d$, so that
\[
F=A_0+A_1Q+\cdots+A_dQ^d,\quad A_d=1.
\]

The comparison of $A_j$ with $F_{Q,j}$ is feasible only when $Q$ is  large enough. For $j>0$, the polynomials $\parj L(P)$ are $\qq$-stable because $\deg\parj L(P)<\deg F$. Take $Q>P$ sufficiently large so  that $\dgq \parj L(P)=0$ for all $0<j<d$. Let us denote
\[
\be_j=\nu\left(\parj L(P)\right)\quad \mbox{for all }\,0< j< d. 
\]
Since $S_Q(\parj L(P))=\{0\}$, Proposition \ref{NS2.4} shows that for all $0< j< d$ we have
\begin{equation}\label{cond2}
\nuq\left(A_j\right)=\nuq\left(\parj L(P)\right)=\be_j\quad \mbox{and}\quad \inq A_j \mbox{ is a unit in }\ggq.
\end{equation}

\emph{From now on, we assume that $Q>P$ is sufficiently large 
to ensure that (\ref{cond2}) holds for all $\,0<j<d$.}\e

%By Corollary \ref{A0}, $\inq A_0=\inq f_{Q,0}$ is a unit; let us denote $\be_0=\nu(A_0)=\nu(f_{Q,0})$. Also, $A_d=f_{Q,d}=1$ is a unit and we denote $\be_d=\nu(A_d)=0$. 

Completing with zero coefficients, consider the canonical $Q$-expansions:	
\[
A_j=A_{j,0}+A_{j,1}Q+\cdots+A_{j,d}Q^d,\qquad 0\le j\le d.
\]

\begin{lemma}\cite[Lemma 4.2]{NS2023}\label{NS4.2}
For all $0\le j\le d$, we have $\nuq(A_j-A_{j,0})+j\gaq >db$.
\end{lemma}
\begin{proof}
Denote $h=h_Q$ and let $L=\sum_{k=0}^dz_kX^k$. Then,
\[
\parj L=\sum_{k= j}^da_kX^{k-j},\qquad A_j=\sum_{k= j}^da_kh^{k-j},\qquad\mbox{where }a_k=\comb{k}{j}z_k. 
\] 
Therefore, 
\[
\nu(a_kh^{k-j})\ge \nu(z_kh^{k-j})=\nu(z_kP^k)-j\ga_P\ge \nu_P(F)-j\ga_P=(d-j)\ga_P.
\]
For all $k>j$, we can proceed as in the proof of Proposition \ref{NS2.4}. Consider the divisions with remainder:
\[
a_kh^{k-j}=a_{k,0}+q_kQ,\quad\deg a_{k,0}<m,\quad  j< k\le d.
\]
For all $k>0$, the inequality in (\ref{basic}) shows that
\[
\nuq(q_kQ)>\nu(a_kh^{k-j})+d(b-\ga_P)\ge db-j\ga_P\ge db-j\gaq.
\]
This proves the lemma.
\end{proof}\e

Let $S\sub I$ be the initial segment of $I$ associated to $H$. Recall that
\[
H=H_S=\left\{x\in\rlx\mid \val(x)> S\right\}.
\]
Consider the group homomorphism determined by truncation by $S$:
\[
\rlx \lra  \rlxs,\qquad x=(x_i)_{i\in I}\longmapsto x_S=(x_i)_{i\in S}.
\]
The kernel of this homomorphism is precisely the convex subgroup $H_S=H$. Thus, it induces an isomorphism of ordered groups between $\rlx/H$ and $\rlxs$, determined by $x+H\mapsto x_S$.  

\emph{From now on, we shall identify $\rlx/H$ with $\rlxs$ without further mention.}

\begin{proposition}\label{aibi}
Consider an arbitrary index $0<j< d$.
\begin{enumerate}
	\item [(a)]$\left(\be_j\right)_S\ge(d-j)b_S$.
	\item[(b)] If $\left(\be_j\right)_S=(d-j)b_S$, then $\nu(F_{Q,j})=\be_j$, $\inq F_{Q,j}=\inq A_j$ and $\nu(F_{Q,j}Q^j)\in\dta^L$.
	\item [(c)] If $\left(\be_j\right)_S>(d-j)b_S$, then $\nu(F_{Q,i})>(d-j)b$. Moreover,
\[
\nu(F_{Q,j}Q^j)>\dta^L\ \sii\ \be_j+j\gaq>\dta^L.
\]	
\end{enumerate}	
\end{proposition}
	
\begin{proof}
Since $\inq A_j$ is a unit for all $0\le j\le d$, none of these elements is divided by $\piq$ in $\gg_Q$. By \cite[Lemma 2.2]{KP}, we have
\[
\nuq(F)=\min\{\be_j+j\,\gaq\mid 0\le j\le d\},
\]
and if $\mathcal{S}=\{0\le j\le d\mid \be_j+j\,\gaq=\nuq(F)\}$, then 
$
\inq F=\sum_{j \in \mathcal{S}}(\inq A_j) \piq^j$. 
By Proposition \ref{Sq} and Theorem \ref{g0gm}, we have $\inq F=\inq A_0+\piq^d$. Hence, 
\[
\be_j+j\,\gaq >\nuq(F)=d\gaq\quad\mbox{for all }0<j<d.
\]
Since this holds for all sufficiently large $Q$, and $b_S=\sup\{(\gaq)_S\mid Q \in \qq\}$ in $\rlxs$
, we deduce that $\left(\be_j\right)_S\ge(d-j)b_S$. This proves (a).\e
 
Since $F=A_0+A_1Q+\cdots+A_dQ^d$, we have:
\begin{equation}\label{set}
F_{Q,j}=\sum\nolimits_{0\le k\le j}A_{j-k,k}.
\end{equation}
For $k=0$, we have $\nu(A_{j,0})=\be_j$ by (\ref{cond2}). For $k>0$, Lemma \ref{NS4.2} shows that $\nu(A_{j-k,k}Q^k)> db-(j-k)\gaq$, so that
\[
\nu(A_{j-k,k})> db-j\gaq\quad \mbox{ for all }0<k\le j.
\]

Suppose  $\left(\be_j\right)_S=(d-j)b_S$. Since $b_S>(\ga_R)_S$ for all $R\in\qq$, we have
\[
\left(db-j\gaq\right)_S=(d-j)b_S+\left(j(b-\gaq)\right)_S>(d-j)b_S=(\be_j)_S.
\]
Hence, $db-j\gaq>\be_j$ and we deduce that $\nu(F_{Q,j})=\be_j$ and $\inq  F_{Q,j}=\inq A_{j,0}=\inq A_j$. 
Moreover, since $\left(\be_j+j\gaq\right)_S<db_S=\sup\{d(\ga_R)_S\mid R\in\qq\}$,  we must have  $\left(\be_j+j\gaq\right)_S<d(\ga_R)_S$ for some $R$. Hence, $\be_j+j\gaq<d\ga_R\in\dta^L$.
This proves (b).\e

Now, suppose  $\left(\be_j\right)_S>(d-j)b_S$. Then, $\be_j>(d-j)b$ and $\nu(A_{j-k,k})> db-j\gaq>(d-j)b$. From the equality (\ref{set}) we deduce that $\nu(F_{Q,j})>(d-j)b$. 

Finally, since $\nu(A_{j-k,k})+j\gaq> db>\dta^L$, the last statement of (c) follows immediately from (\ref{set}) too.
\end{proof}\e

In this way, we obtain sharp information about the canonical $Q$-expansions of $F$, for $Q$ sufficiently large. The Newton polygon $N_{\nuq,Q}(F)$ has the following shape. 

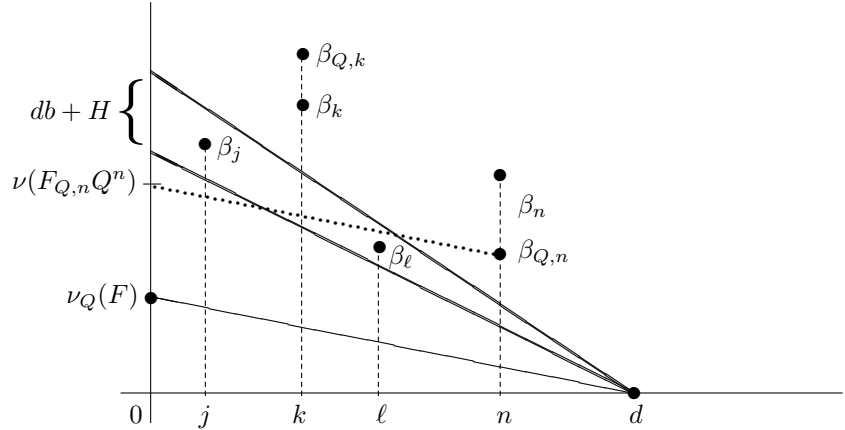
\begin{center}
	\setlength{\unitlength}{4mm}
\begin{picture}(22,14.5)
\put(13.35,5.35){$\bullet$}\put(13.35,8){$\bullet$}
\put(9.35,5.6){$\bullet$}
\put(3.55,9){$\bullet$}
\put(6.8,10.3){$\bullet$}
\put(6.8,12){$\bullet$}
\put(1.75,3.9){$\bullet$}\put(17.8,0.75){$\bullet$}
\put(1,1){\line(1,0){24}}
%\put(-0.25,13.15){\line(1,0){0.6}}
\put(1.75,7.95){\line(1,0){0.6}}
\put(2,0){\line(0,1){14}}
\put(18,1){\line(-3,2){16}}\put(18,1.05){\line(-3,2){16}}
\put(18,1){\line(-2,1){16}}\put(18,1.05){\line(-2,1){16}}
\put(18,1){\line(-5,1){16}}
\multiput(2,7.6)(0.25,-0.05){46}{$\cdot$}
\multiput(7,.9)(0,.3){38}{\vrule height2pt}
\multiput(3.8,.9)(0,.3){29}{\vrule height2pt}
\multiput(9.55,.9)(0,.3){17}{\vrule height2pt}
\multiput(13.6,.9)(0,.3){25}{\vrule height2pt}
\put(.8,9.8){\begin{footnotesize}$\mbox{\huge$\{$}$                \end{footnotesize}}
\put(-2,10.1){\begin{footnotesize}$db+H$\end{footnotesize}}
\put(-.8,4){\begin{footnotesize}$\nuq(F)$\end{footnotesize}}
\put(-2.6,7.8){\begin{footnotesize}$\nu(F_{Q,n}Q^n)$\end{footnotesize}}
\put(17.9,0){\begin{footnotesize}$d$\end{footnotesize}}
\put(3.6,0){\begin{footnotesize}$j$\end{footnotesize}}
\put(6.7,0){\begin{footnotesize}$k$\end{footnotesize}}
\put(9.5,0){\begin{footnotesize}$\ell$\end{footnotesize}}
\put(13.5,0){\begin{footnotesize}$n$\end{footnotesize}}
\put(1.3,0){\begin{footnotesize}$0$\end{footnotesize}}
%\multiput(-.2,3)(.3,0){31}{\hbox to 2pt{\hrulefill }}
\put(14.2,5.5){\begin{footnotesize}$\be_{Q,n}$\end{footnotesize}}
\put(7.5,12){\begin{footnotesize}$\be_{Q,k}$\end{footnotesize}}
\put(4.2,8.9){\begin{footnotesize}$\be_j$\end{footnotesize}}
\put(9.9,5.3){\begin{footnotesize}$\bel$\end{footnotesize}}
\put(7.5,10.3){\begin{footnotesize}$\be_k$\end{footnotesize}}
\put(14.2,7){\begin{footnotesize}$\be_n$\end{footnotesize}}
	\end{picture}
\end{center}\bs

Except for the point $(0,\nuq(F))$, the rest of the points $(j,\be_{Q,j})$ lie on or above the ``limit region" determined by all lines joining $(d,0)$ with all points in the set $db+H$. The line joining $(0,\nuq(F))$ with $(d,0)$ has slope $-\gaq$. \e

\noindent{\bf Definition. }{\it Define the \textbf{limit index set} $B'(F)$ as the set containing the indices $j$ for which the points $(j,\be_j)$ fall in the limit region:}
\[
B'(F)=\{j\in(0,d)\mid (\be_j)_S=(d-j)b_S\}.
\]	

As the picture above shows, there might be points $(n,\nu(F_{Q,n}))$ lying above the limit region, such that $\nu(F_{Q,n}Q^n)<db+H$. However, since the values $\be_1,\dots,\be_{d-1}$ are constant, we may avoid this phenomenon by taking $Q$ sufficiently large.

\begin{lemma}\label{cond3}
	There exists $Q_0\in\qq$ such that $\be_j+j\gaq >\dta^L$ simultaneously for all $j\not\in B'(F)$ and all $Q\ge Q_0$.
\end{lemma} 

\begin{proof}
	For $j\not\in B'(F)$, we have $\left(db-\be_j\right)_S<jb_S$. Hence, there exists $Q_0\in\qq$ such that $\left(db-\be_j\right)_S<j\left(\gaq\right)_S$ for all $Q\ge Q_0$. This implies $\be_j+j\gaq>db>\dta^L$ for all $Q\ge Q_0$.
	Since $(0,d)\setminus B'(F)$ is a finite set, we may find a common $Q_0$ such that this holds for all $j\not\in B'(F)$ and all $Q\ge Q_0$ simultaneously.
\end{proof}%\e

\begin{corollary}\label{welldefined}
	For all $F,G\in\kpi(\qq)$, we have $B'(F)=B'(G)$.
	
	Moreover, $\inq F_{Q,j}=\inq G_{Q,j}\,$ for all $0<j<d$, if \,$Q$ is large enough.
\end{corollary}

\begin{proof}
Take any $Q\ge Q_0$, where $Q_0$ is large enough so that Lemma \ref{cond3} holds simultaneously for $F$ and $G$. Then, for all $0<j<d$, we have
\[
 j\in B'(F)\quad\sii\quad \nu(F_{Q,j}Q^j)\in \dta^L,
\]
and a completely analogous characterization holds for $B'(G)$.

By Lema \ref{MinAug}, $\nu(F-G)>\dta^L$. Since $\deg(F-G)<\deg(F)$, the polynomial $F-G$ is $\qq$-stable and we may assume $\nuq(F-G)=\nu(F-G)>\dta^L$ for $Q$ large enough. This implies $\nu((F_{Q,j}-G_{Q,j})Q^j)=\nuq((F_{Q,j}-G_{Q,j})Q^j)\ge \nuq(F-G)>\dta^L$, for all $0\le j\le d$. This proves $B'(F)=B'(G)$ by the above criteria.

Finally, the equality $\inq F_{Q,j}=\inq G_{Q,j}$ follows from Lemma \ref{MinAug} too.
\end{proof}\e

Let us denote the limit index set simply by $B'=B'(F)$.

As a consequence of  Proposition \ref{aibi} and Lemma \ref{cond3}, we see that $B'=\bct$ and $\bct\cup J=\{1,\dots,d\}$. In particular, the set $B'$ did not depend on the choice of the key polynomial $P$, from which the polynomial $L\in\kx_m[X]$ was constructed. 

We can summarize our achievement as follows.

\begin{theorem}\label{mainVB}
If $\qq$ is VB, then $B=\{0\}\cup \bct$.
\end{theorem}

Finally, as an immediate consequence of \cite[Lemma 4.1]{NS2023} applied to the quotient group $\rlx/H$, we derive the following result.

\begin{corollary}\label{powerp2}
Every index in $\bct$ is a power of the exponent characteristic of $v$.
\end{corollary}

From Theorem \ref{mainVB} and Proposition \ref{existence}, we obtain the following result, which generalizes \cite[Proposition 9.2]{hos}, \cite[Theorem 1.1]{NS2023} and \cite[Theorem 1.2]{Nov14}.

\begin{theorem}\label{main}
Let $\qq$ be a VB set of key polynomials for $\nu$. Let $F\in\kpi(\qq)$ be a limit key polynomial. Then, for $Q$ sufficiently large, the polynomial
\[
G=F_{Q,0}+\sum\nolimits_{j\in\bct}F_{Q,j}Q^j
\] 
is a minimal limit key polynomial for $\nu$.
\end{theorem} 

The condition $B=\{0\}\cup\bct$ implies that the minimal limit key polynomials satisfy a stronger property. Namely, if $F=F_{Q,B}$ is a minimal limit key polynomial, then for all $\ell\in B$, the polynomial $F_{R,B\setminus\{\ell\}}$ is not a limit key polynomial for all   $R\ge Q$, and not only for a cofinal family of $\qq$.

%Let $\ii\sub\N$ be the subset determined by  $\bct=\{p^j\mid j\in \ii\}$. The \enric{minimal} limit key polynomial of Theorem \ref{main} can be rewritten as\[G=F_{Q,0}+\sum\nolimits_{j\in \ii}F_{Q,p^j}Q^{p^j}.\]

\section{The vertically unbounded case}\label{secVU}
We keep with the notation of the previous section.
%In this section, we deal with the VU case in general, including the unbounded case.

We say that $\qq$ is \textbf{vertically unbounded (VU)}, when it is not vertically bounded. In this case, Lemma \ref{VS} and Proposition \ref{VS=VB} show that the cut $\dta$ is equal to $\left(a+H\right)^+$ for some $a\in\g$ and some convex subgroup $H\ne\{0\}$.
Obviously, $H$ is the invariance group of $\dta$. Note that $H$ is the invariance group too, of the  cuts $n\gaqq$ for all $n\in\Z$. 

In the VU case, we cannot guarantee that $\deg(F)=md$ anymore. Actually, the relative degree $D$ of $F$ with respect to $Q$ will be usually strictly larger than $d$.

For all $\ell\in \bct$, we denote by $\bel$ the constant value of $\be_{Q,\ell}$  for $Q$ large enough.

Take $Q_0\in\qq$ large enough, so that the following conditions hold  for all $Q\ge Q_0$.
\ars{1.2}\e

\begin{tabular}{l}
$\bullet\quad S_Q(F)=\{0,d\}$ and $I(Q)=\{s_\infty\}$.\\
$\bullet\quad \be_{Q,\ell}=\bel$ for all $\ell\in \bct$. \\
$\bullet\quad \be_{Q,\ell}+\ell\gaq>\dta^L$ for all $\ell\in J$. \\
$\bullet\quad \nuq(F)+H$ is  a final segment of $\dta^L$.   
\end{tabular}\ars1

\subsection{Minimal index in $\bct$.}

\begin{definition}\label{definunbiuubs}
 For all $\ell\in[0,D]\cap\Z$, let $$\bbl=\delta^L-\ell\gamma_{\rm min},$$i.e., $\bbl$ is the minimal initial segment of $\g$ containing all $\be\in\g$ such that $\be+\ell\gaq\in \dta^L$ for some $Q\in\qq$.

 Also, let $\bunb\sub B$ be the subset of all $\ell\in B$ such that for all $\be\in\bbl$ we have  
 \[
 \be_{Q,\ell}\ge\be \quad\mbox{ for all \,$Q$ in a final segment of $\qq$}. 
 \]
\end{definition}

\begin{lemma}\label{dD}
For the indices $0,d,D$ we have $0\in \bunb$ and $d,D\in \bct$. 
\end{lemma}

\begin{proof}
For all $Q\ge Q_0$, we have $\be_{Q,0}=\nuq(F)$. Since $\bb_0=\dta^L$, we have $0\in \bunb$ by the definition of $\dta$.  Now, the coefficient $F_{Q,D}$ is independent of $Q$; hence, $D\not\in J$ by Proposition \ref{existence} and the minimality of $\deg(F)$ among all $\qq$-unstable polynomials. 
Finally, $d\in \bct$ follows from Proposition \ref{ellstable}
and the definition of $\dta$. 
\end{proof}\e

\begin{lemma}\label{valueINdelta}
Take  $\be\in\bbl$ for some $\ell\in[0,D]\cap\Z$. Then, $\be+H\sub \bbl$. In particular,
$\be+\ell\gaq\in\dta^L$ for all $ Q\in\qq$. 
\end{lemma}

\begin{proof}
Take any $P\in\qq$ such that $\be+\ell\ga_P\in \dta^L$. Then,
$ \be+H+\ell\ga_P\sub \dta^L+H=\dta^L$.
Also, for all $Q\in\qq$, we have
\[
\be+\ell \gaq=\be+\ell\ga_P+\ell(\gaq-\ga_P)\in\be+\ell\ga_P+H\sub \dta^L.
\]
This ends the proof. 
\end{proof}%\e

\begin{definition}
 We say that $\be\in\bbl$ is \textbf{sharp} if $\be+H$ is a final segment of $\bbl$.
\end{definition}

\begin{lemma}\label{wholecut}
Suppose that $\be\in\bbl$ is sharp. Then, for all $k\in[0,D]\cap\Z$ and  all $Q\in\qq$, the element $\be+(\ell-k)\gaq\in\bb_k$ is sharp. 
\end{lemma}

\begin{proof}
Take $k\in[0,D]\cap\Z$ and  $Q\in\qq$. Suppose that $\be+(\ell-k)\gaq<\al$ for some $\al\in\bb_k$. Let $\be'=\al-(\ell-k)\gaq>\be$. Obviously, $\be'$ belongs to $\bbl$ because $\be'+\ell\gaq=\al+k\gaq\in\dta^L$. Since $\be$ is sharp, there exists $h\in H$ such that $\be'=\be+h$. In particular, $\al=\be+(\ell-k)\gaq+h$. This proves that $\be+(\ell-k)\gaq$ is sharp.
\end{proof}\e

\begin{lemma}\label{allsharp}
Take $\ell\in[0,D]\cap\Z$ and $Q\ge Q_0$. If $\be_{Q,\ell}\in \bbl$, then $\be_{Q,\ell}$ is sharp.
\end{lemma}

\begin{proof}
By our assumptions on $Q_0$, the value $\nuq(F)\in\bb_0$ is sharp.
Suppose that $\be_{Q,\ell}\in \bbl$.
Since $\nuq(F)\le  \be:=\be_{Q,\ell}+\ell\gaq\in\dta^L$, the element $\be\in\bb_0$ is sharp too. By Lemma \ref{wholecut}, $\be_{Q,\ell}$ is sharp in $\bbl$.
\end{proof}%\e

\begin{proposition}\label{largeslope}
Let $\al\in\bbl$, $\be\in\bb_k$, for some  $0\le k<\ell\le D$. Let $-\la$ be  the slope of the segment joining $(\ell,\al)$ and $(k,\be)$.  If $\al$ is sharp, then  $\gar>\la$ for some $R\in\qq$.
\end{proposition}

\begin{proof}
Suppose that $\la>\gar$ for all $R\in\qq$. Let us show that this implies $\be\not \in\bb_k$, which contradicts our assumptions.

Take any $Q\in \qq$. Since $\al$ is sharp, Lemma \ref{wholecut} shows that $\al+\ell\gaq+H$ is a final segment of $\dta^L$. Our aim is to show that
$\be+k\gaq>\dta^L$, or equivalently,
\[
\be+k\gaq>\al+\ell\gaq+H.
\]
Since $-\la=(\be-\al)/(k-\ell)$, we have $\be=\al+(\ell-k)\la$. Thus, our aim is equivalent to
$$%\begin{equation}\label{aim3}
(\ell-k)\la>(\ell-k)\gaq+H. 
$$%\end{equation}
This follows from $\la>\gaqq^L$, since $H$ is the invariance group of the cut $(\ell-k)\gaqq$
\end{proof}\e

\begin{theorem}\label{minB1}
$d=\min(B\setminus\bunb)=\min(\bct)$.
\end{theorem}

\begin{proof}
Take any $k\in B\setminus\bunb$; that is, for some $\be\in \bb_k$ we have
$\be_{S,k}\le\be$ for all $S$ in a cofinal family $\mathcal{S}$ of $\qq$. Let $-\la$ be the slope of the segment joining $(d,\bed)$ with $(k,\be)$. Let us see that the condition $k<d$ leads to a contradiction. Indeed, by Proposition \ref{Sq}, $\la>\gas$ for all $S\in \mathcal{S}$, because the point $(k,\be_{S,k})$ lies strictly above the segment of slope $-\gas$, joining $(d,\bed)$ with $(0,\nu_S(F))$.   
%\begin{figure}%[h]	\caption{Newton polygon $N=\np(g)$ of $g\in \kx$. }\label{figNmodel}
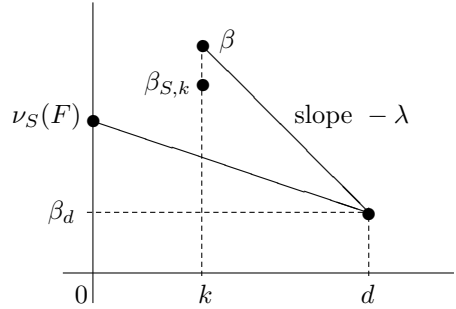
\begin{center}
\setlength{\unitlength}{4mm}
\begin{picture}(12,9.5)
\put(8.9,1.7){$\bullet$}\put(3.4,7.3){$\bullet$}
\put(3.4,6){$\bullet$}\put(-0.25,4.8){$\bullet$}
\put(-1,0){\line(1,0){13}}\put(0,-1){\line(0,1){10}}
\put(0,5.05){\line(3,-1){9}}\put(3.6,7.6){\line(1,-1){5.6}}
\multiput(9.15,-0.1)(0,.3){8}{\vrule height2pt}
\multiput(3.6,-0.1)(0,.3){26}{\vrule height2pt}
%\multiput(8,.9)(0,.25){9}{\vrule height2pt}
\put(-2.7,5){\begin{footnotesize}$\nu_S(F)$\end{footnotesize}}
\put(3.5,-1){\begin{footnotesize}$k$\end{footnotesize}}
\put(8.9,-1){\begin{footnotesize}$d$\end{footnotesize}}
\put(-1.5,1.8){\begin{footnotesize}$ \bed$\end{footnotesize}}
\put(4.2,7.5){\begin{footnotesize}$\be$\end{footnotesize}}
\put(1.7,6.1){\begin{footnotesize}$\be_{S,k}$\end{footnotesize}}
\put(6.7,5){\begin{footnotesize}$\mbox{slope }-\la$\end{footnotesize}}
\put(-.6,-1){\begin{footnotesize}$0$\end{footnotesize}}
\multiput(-.2,2)(.3,0){31}{\hbox to 2pt{\hrulefill }}
%\put(12.7,0.45){\begin{footnotesize}$\mathfrak{n}$\end{footnotesize}}
%\put(11.6,2.5){\begin{footnotesize}$\mathfrak{n}$\end{footnotesize}}
\end{picture}
\end{center}\bs\e
%\end{figure}
Hence, $\la>\gaq$ for all $Q\in\qq$. This contradicts Proposition \ref{largeslope}, because $\bed\in\bb_d$ is sharp, by Lemma \ref{allsharp}.
Since $\bct\sub \left(B\setminus\bunb\right)$, we have $d=\min(\bct)$ as well. 
\end{proof}\e

Let us mention another consequence of Proposition \ref{largeslope}. 
By Lemma \ref{dD} and Theorem
 \ref{minB1}, we have
\[
\bct=\left\{d=\ell_0<\ell_1<\cdots<\ell_t=D\right\}.
\]
In \cite[Lemma 3.11]{KuhlAT}, Kuhlmann uses a strong version of a classical result by Kaplansky. Namely,  there exists a permutation $\sg\in S_t$ 
and $Q_1\ge Q_0$ such that  
\begin{equation}\label{kaplansky}
\be_{\ell_{\sg(1)}}+\ell_{\sg(1)}\gaq<\cdots <\be_{\ell_{\sg(t)}}+\ell_{\sg(t)}\gaq\quad \mbox{ for all }\ Q\ge Q_1.
\end{equation}

\begin{corollary}\label{sg=1}
In the VU case, the inequalities of (\ref{kaplansky}) can occur only for $\sg=1$.

In particular, for every pair $\ell_i<\ell_j$ in $\bct$, the slope $-\la$  of the segment joining $(\ell_i,\be_{\ell_i})$ with $(\ell_j,\be_{\ell_j})$ satisfes $\la<\ga_{Q_1}$. 
\end{corollary}

\begin{proof}
 Suppose $\ell_i<\ell_j$ and  $\be_{\ell_i}+\ell_i\gaq>\be_{\ell_j}+\ell_j\gaq$, for all $Q\ge Q_1$. Then,
 \[
  \la=(\be_{\ell_i}-\be_{\ell_j})/(\ell_j-\ell_i)>\gaq\quad \mbox{ for all }\ Q\ge Q_1.
 \]
This contradicts Proposition \ref{largeslope}. In particular, $\la<\ga_{Q_1}$.
\end{proof}\e

\subsection{Comparison of $Q$-expansions}\label{subsecComp}
In the VU case, the comparison of $Q$-expan\-sions of $F$ is much more sophisticated than in the VB case. We shall be able to describe the set $B$ only in some particular cases. Let us denote  $\gmin=\ga_{Q_{\op{min}}}$.

\begin{lemma}\label{hhh}
Suppose that $\si=1$. Let $a=h_1\cdots h_n$ for some $h_1,\dots,h_n\in \kx_m$.   
For some $R\ge Q_0$, let $[a]_i$ be the $i$-th coefficient of the $R$-expansion of $a$. Then, \[\nu\left([a]_i\right)>\nu(a)-i\gmin\quad \mbox{ for all }\,i>0.\] 
\end{lemma}

\begin{proof}
Let us use induction on $n$. For $n=1$, the statement is empty. For $n=2$, it follows from Lemma \ref{basic00}:
\[
a=[a]_0+[a]_1R\ \imp  \ \nu\left([a]_0\right)=\nu(a),\quad \nu\left([a]_1R\right)>\nu(a)+\ep_R-\epmin.
\]
Since $\si=1$, equation (\ref{epga}) shows that $\ep_R-\epmin=\gar-\gmin$. Thus,
\[
\nu\left([a]_1\right)>\nu(a)-\gar+\ep_R-\epmin=\nu(a)-\gmin.
\]

Take $n>2$ and assume that the statement holds for $n-1$. Then,
\[
h_1\cdots h_{n-1}=e_0+e_1R+\cdots +e_{n-2}R^{n-2}, 
\]
where the coefficients $e_1,\dots,e_{n-2}\in\kx_m$ satisfy 
\begin{equation}\label{n-1}
\nu(e_i)>\nu(h_1\cdots h_{n-1})-i\gmin \quad \mbox{ for all }\ i>0. 
\end{equation}
 From the equality $a=e_0h_n+e_1h_nR+\cdots +e_{n-2}h_nR^{n-2}$,
we deduce 
\[
 [a]_1=[e_0h_n]_1+[e_1h_n]_0,\dots, [a]_{n-2}=[e_{n-3}h_n]_1+[e_{n-2}h_n]_0,\ [a]_{n-1}=[e_{n-2}h_n]_1.
\]
The statement follows easily from (\ref{n-1}) and the case $n=2$.
\end{proof}\e

Consider $R>Q\ge Q_1$ in $\qq$ and let the corresponding expansions of $F$ be: 
\[
F=a_0+a_1Q+\cdots+a_DQ^D=b_0+b_1R+\cdots+b_DR^D.
\]

Write $Q=R+h$ for some $h\in \kx_m$. Recall that $\nu(h)=\nu(Q)=\gaq$. With the notation of Lemma \ref{hhh}, we have
\begin{equation}\label{bk}
 b_k=\sum_{k\le j}A_{k,j},\qquad A_{k,j}=\sum_{i=0}^k\comb{j}{k-i}\left[a_j h^{j-k+i}\right]_i.
\end{equation}

By Lemma \ref{hhh}, for all $i>0$ we have
\begin{equation}\label{estfori}
\nu\left(\left[a_j h^{j-k+i}\right]_i\right)> \be_{Q,j}+(j-k)\gaq+i\left(\gaq-\gmin\right).
\end{equation}

The following result is an immediate consequence of (\ref{bk}) and (\ref{estfori}).

\begin{lemma}\label{estimations}
If $\si=1$ and $0<k\le\ell\le D$, then $\nu\left(A_{k,\ell}\right)\ge \be_{Q,\ell}+(\ell-k)\gaq$ and equality holds if and only if $p\nmid\comb{\ell}{k}$. In particular, $\nu\left(A_{k,k}\right)=\nu(a_k)$.
\end{lemma}

\begin{theorem}\label{B1=B0}
 If $\qq$ is VU and $\si=1$, then $B=\bct\sqcup \bunb$.
\end{theorem}

\begin{proof}
Suppose $\bct\sqcup \bunb\subsetneq B$ and take $k\in B$ maximal such that $k\not\in \bct\sqcup \bunb$. 
Then, for some $\be\in\bb_k$ we have $\be_{S,k}\le\be$ for all $S$ in some cofinal family $\mathcal{S}$ of $\qq$. 

Take $Q\in\mathcal{S}$ such that $Q\ge Q_1$. Let us estimate the values $\nu\left(A_{k,j}\right)$ for all $j \ge k$.

By Lemma \ref{estimations},  $\nu\left(A_{k,k}\right)=\nu(a_k)$. Now, we claim that there exists $Q\in\mathcal{S}$ large enough, such that $\nu\left(A_{k,j}\right)>\be$ for all $j>k$.
By (\ref{bk}), these inequalities imply $\nu(b_k)=\nu(a_k)$ for all $R>Q$. Therefore, $k\in \bct$, contradicting our assumption.

Let us fix an index $j>k$. By Lemma \ref{estimations}, in order to prove that  $\nu\left(A_{k,j}\right)>\be$, it suffices to show that
\begin{equation}\label{aimj}
\be_{Q,j}+(j-k)\gaq>\be\quad \mbox{ for all \,$Q$ large enough}.
\end{equation}

If $j\in J$, then $\be_{Q,j}+j\gaq>\dta^L$ for all $Q$ large enough. Since $\be+k\gaq\in\dta^L$, the inequality in (\ref{aimj}) holds. 

Suppose $j\in \bunb$. Since $\be+(k-j)\gmin\in\bb_j$, we have
\[
\be_{Q,j}>\be+(k-j)\gmin\quad \mbox{ for all \,$Q$ large enough}.
\]
For any such $Q$, (\ref{aimj}) holds. Indeed, since $j>k$, we have
\[
 \be_{Q,j}+(j-k)\gaq>\be+(k-j)\gmin+(j-k)\gaq=\be+(j-k)(\gaq-\gmin)>\be.
\]

Finally, suppose $j\in \bct$. By  Lemma \ref{allsharp}, $\be_j+(j-k)\ga_{Q_1}$ is sharp in $\bb_k$. Hence, $\be=\be_j+(j-k)\ga_{Q_1}+h$, for some $h\in H$. Now, any $Q\in\mathcal{S}$ such that $(j-k)(\gaq-\ga_{Q_1})>h$ satisfies (\ref{aimj}). 
This ends the proof of the theorem, because there exists $Q\in\mathcal{S}$ satisfying (\ref{aimj}) simultaneously for the finite number of indices $j>k$. 
\end{proof}\e

Let us end this section with a remark about the shape of the Newton polygons $N_{\nuq,Q}(F)$, displayed in Figure \ref{figNPult}. We recall that $N_{\nuq,Q}(F)$ is the lower convex hull of the set $\{(\ell,\be_{Q,\ell})\mid 0\le \ell\le D\}$ in the $\Q$-vector space $\Q\times \g$.

\begin{figure}%[h]	
\caption{Newton polygon $N_{\nuq,Q}(F)$. The slope of the left-most side is $-\gaq$. All other slopes are of the form $-\la$, with $\la<\gaq$.}\label{figNPult}
\begin{center}
\setlength{\unitlength}{3mm}
\begin{picture}(20,14)
\put(15,9.7){$\bullet$}\put(13,9){$\bullet$}\put(12,3.75){$\bullet$}\put(2.1,4.5){$\bullet$}\put(4,7){$\bullet$}\put(7.4,1.45){$\bullet$}
\put(6,0.8){$\bullet$}\put(1,10){$\bullet$}\put(-0.3,11.75){$\bullet$}
\put(-1,3){\line(1,0){20}}\put(0,0){\line(0,1){13.5}}
\put(0,12.1){\line(1,-3){2.4}}\put(2.25,5){\line(1,-1){4}}
\put(6.2,1){\line(2,1){6}}
\put(12.25,4){\line(1,2){3}}
\multiput(2.4,3)(0,.3){6}{\vrule height2pt}
\multiput(15.25,3)(0,.25){27}{\vrule height2pt}
\put(-3.6,11.9){\begin{footnotesize}$\nuq(F)$\end{footnotesize}}
\put(2.1,1.7){\begin{footnotesize}$d$\end{footnotesize}}
\put(14.8,1.7){\begin{footnotesize}$D$\end{footnotesize}}
\put(-1.8,4.6){\begin{footnotesize}$ \bed$\end{footnotesize}}
\put(15.9,10){\begin{footnotesize}$ \be_D$\end{footnotesize}}
\put(-.8,1.7){\begin{footnotesize}$0$\end{footnotesize}}
\multiput(-.2,4.8)(.3,0){8}{\hbox to 2pt{\hrulefill }}
%\put(12.7,0.45){\begin{footnotesize}$\mathfrak{n}$\end{footnotesize}}
%\put(11.6,2.5){\begin{footnotesize}$\mathfrak{n}$\end{footnotesize}}
\end{picture}
\end{center}
\end{figure}
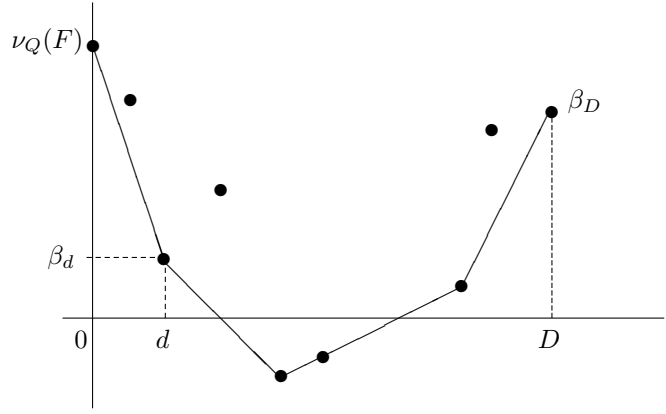
\bs

\begin{proposition}\label{NPconstant}
 Let $F\in\kpi(\qq)$. If $\si=1$, then the  partial Newton polygon $N_{\nuq,Q}(F)\cap [d,D]$ and the set
 \[
N_{d,D}:=\left\{k\in[d,D]\cap\N\mid (k,\be_{Q,k})\in N_{\nuq,Q}(F)\right\}
 \]
are both ultimately constant. In particular, $N_{d,D}\subseteq \bct$.
\end{proposition}

\begin{proof}
Take any $Q\ge Q_1$ and denote $N=N_{\nuq,Q}(F)$.
If $d=D$, then the statement (stability of $\be_{Q,D}$) follows from Proposition \ref{ellstable}.
Suppose $d<D$. 

Consider any side $S$ of $N$ whose end vertices $(\ell,\al)$, $(L,\be)$ satisfy $d\le \ell< L\le D$ and $L\in \bct$. In particular, $\be:=\be_L$ is the stable value of all $\be_{R,L}$ for $R\ge Q$.  
Let $-\la=(\be-\al)/(L-\ell)$ be the slope of this side. By the lower convexity of $N$, we have $-\gaq<-\la$, since $-\gaq$ is the slope of the left-most side of $N$.% for all $Q\ge Q_1$.

Take any index $k$ such that $\ell\le k<L$. By the lower convexity of $N$, we  have 
\[
\be_{Q,j}+(j-k)\gaq>\be_{Q,j}+(j-k)\la\ge\be+(L-k)\la \quad \mbox{ for all }k\le j\le D,
\]
and equality holds if and only if $j=k$ and the point $(k,\be_{Q,k})$ lies on $N$. \bs

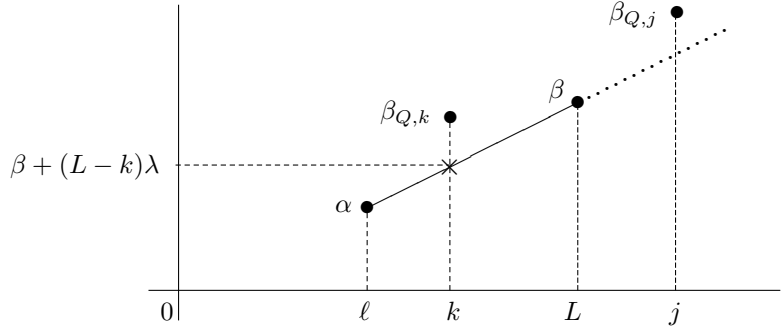
\begin{center}
\setlength{\unitlength}{3mm}
\begin{picture}(20,11)
\put(13,7){$\bullet$}\put(6,3.5){$\bullet$}
\put(8.6,4.85){$\times$}\put(8.77,6.5){$\bullet$}
\put(16.3,10){$\bullet$}
\put(-1,1){\line(1,0){21}}\put(0,0){\line(0,1){10.5}}
\put(6.2,3.7){\line(2,1){7}}
\multiput(13.5,7.35)(0.3,.15){16}{.}
\multiput(9,1)(0,.3){20}{\vrule height2pt}
\multiput(6.25,1)(0,.3){10}{\vrule height2pt}
\multiput(13.25,1)(0,.25){26}{\vrule height2pt}
\multiput(16.5,1)(0,.25){37}{\vrule height2pt}
\put(16.3,-0.2){\begin{footnotesize}$j$\end{footnotesize}}
\put(12.8,-0.2){\begin{footnotesize}$L$\end{footnotesize}}
\put(6,-0.2){\begin{footnotesize}$\ell$\end{footnotesize}}
\put(8.9,-0.2){\begin{footnotesize}$k$\end{footnotesize}}
\put(5,3.6){\begin{footnotesize}$ \al$\end{footnotesize}}
\put(12.2,7.5){\begin{footnotesize}$ \be$\end{footnotesize}}
\put(6.4,6.7){\begin{footnotesize}$\be_{Q,k}$\end{footnotesize}}
\put(14,10){\begin{footnotesize}$\be_{Q,j}$\end{footnotesize}}
\put(-7,4.9){\begin{footnotesize}$ \be+(L-k)\la$\end{footnotesize}}
\put(-0.9,-0.2){\begin{footnotesize}$0$\end{footnotesize}}
\multiput(-.1,5.15)(.3,0){31}{\hbox to 2pt{\hrulefill }}
\end{picture}
\end{center}\bs

By Lemma \ref{estimations}, we have $\nu(A_{k,k})=\nu(a_k)=\be_{Q,k}$ and $\nu(A_{k,j})>\be+(L-k)\la$. Hence, for all $R>Q$ we have $\be_{R,k}\ge \be+(L-k)\la$, and equality holds if and only  if the point $(k,\be_{Q,k})$ lies on $N$.
As a consequence, the side $S$ is stable for all $R\ge Q$ and the points lying on $S$ are stable too. In particular, $\ell\in\bct$.

A recurrent application of this argument proves the proposition.
\end{proof}

\subsection{The case $v(p)\in H$}

In this section, we assume that $v(p)$ belongs to $H$; that is, $\chr(K)=0$ and $v(\Q)\sub H$. Our aim is to prove the following result.

\begin{theorem}\label{Jempty}
 If $\qq$ is VU, $\si=1$ and $v(p)\in H$, then $J=\emptyset$.%$[0,D]\cap\Z=\bct\cup \bunb$.
\end{theorem}

The proof of Theorem \ref{Jempty} requires some previous considerations.
For any pair of indices $k<j$, denote
\begin{equation}\label{alkj}
\al_{k,j}=\nu(\comb{j}{k})\in H.  
\end{equation}

\begin{lemma}\label{Akl}
Suppose that $\si=1$,  $v(p)\in H$ and $0<k\le j\le D$.
For any given $\al\in H$  we have
\begin{equation}\label{wanted}
 \nu\left(A_{k,j}-\comb{j}{k}a_j h^{j-k}\right)>\al+\be_{Q,j}+(j-k)\gaq\quad\mbox{ for  all \,$Q$  large enough}. 
\end{equation}

In particular, $\nu\left(A_{k,j}\right)=\al_{k,j}+\be_{Q,j}+(j-k)\gaq$, for  all \,$Q$  large enough. 
\end{lemma}

\begin{proof}
Since $\al\in H$, we have $\gaq-\gmin>\al$ for $Q$ large enough. Thus, the following inequality is a consequence of (\ref{estfori}):
\[
 \nu\left(A_{k,j}-\comb{j}{k}\left[a_j h^{j-k}\right]_0\right)>\al+\be_{Q,j}+(j-k)\gaq.
\]

On the other hand, by  Lemma \ref{basic00} and  (\ref{epga}), we have
\[
\nu\left(a_j h^{j-k}-\left[a_j h^{j-k}\right]_0\right)>\nu\left(a_j h^{j-k}\right)+\gar-\gmin>\nu\left(a_j h^{j-k}\right)+\al.
\]
This proves (\ref{wanted}). For the second conclusion, it suffices to take $\al>\al_{k,j}$.
\end{proof}\e

Theorem \ref{Jempty} follows from a recurrent application of the following result.

\begin{lemma}\label{B0B2}
Suppose $\si=1$  and  $v(p)\in H$. Take $0<k<D$ such that $(k,D]\cap\Z\sub B$. 
Then, $k\in B$. 
Moreover,  for $\ell=\min((k,D]\cap \bct)$, consider the natural numbers $n_{\ell-1},\dots,n_k$ defined recursively by $n_{\ell-1}=\ell$ and 
\[
n_j=\comb{\ell}{j}-\comb{\ell-1}{j}n_{\ell-1}-\cdots-\comb{j+1}{j}n_{j+1}. 
\]
Let $\al:=\max\{\nu(n_j)\mid k\le j\le \ell\}$. If $k\in \bunb$, then  for all $Q$ large enough we have
\begin{equation}\label{claim}
\nu\left(a_k+n_ka_\ell h^{\ell-k}\right)>\al+\bel+(\ell-k)\gaq\quad\mbox{   for all \,$R>Q$ large enough}.  
\end{equation}
In particular, $\nu(a_k)=\nu(n_k)+\bel+(\ell-k)\gaq$.
 \end{lemma}

\begin{proof}
We claim that for every $\ell<j\le D$, we have
\begin{equation}\label{ell<j}
\be_{Q,j}+j\gaq> \al+\bel+\ell\gaq\quad\mbox{ for all $Q$ large enough}.
\end{equation}

Indeed, let us first assume $j\in \bunb$. Since $\bel\in\bbl$, Lemma \ref{valueINdelta} shows that $\bel+(\ell-j)\gmin+\al\,$ belongs to $\bb_j$.  Since $j\in \bunb$, we have $\be_{Q,j}> \al+\bel+(\ell-j)\gmin$ for all $Q$ large enough.
For any such $Q$, we have
\begin{align*}
\be_{Q,j}+j\gaq\,&>\ \al+\bel+(\ell-j)\gmin+j\gaq\\
&=\ \al+\bel+\ell\gaq+(j-\ell)(\gaq-\gmin)>\  \al+\bel+\ell\gaq.
\end{align*}

Now, suppose $j\in \bct$, $j>\ell$. Corollary \ref{sg=1} shows that
\[
 \bel+\ell\gaq< \be_j+j\gaq\quad\mbox{ for all }\,Q\ge Q_1.
\]
The elements in both sides of the inequality belong to $\bb_0$ and they are sharp by Lemma \ref{allsharp}. Hence, their difference belongs to $H$. Since  
$\al$ belongs to $H$ too, we may find $Q$ large enough so that
\[
(\be_j+j\ga_{Q_1})-(\bel+\ell\ga_{Q_1})+(j-\ell)(\gaq-\ga_{Q_1})> \al.
\]
This ends the proof of (\ref{ell<j}).

As a consequence of (\ref{ell<j}) and Lemma \ref{Akl}, we may find $Q_2\in\qq$ large enough so that for all $Q\ge Q_2$
we have simultaneously:
\begin{equation}\label{always}\ars{1.4}
\begin{array}{ll}
\bullet \ &\nu\left(A_{k,j}\right)=\al_{k,j}+\be_{Q,_j}+(j-k)\gaq \ \mbox{ for all }\,j\ge k.\\
\bullet \ &\nu\left(a_j h^{j-k}-\left[a_j h^{j-k}\right]_0\right)>\al+\be_{Q,_j}+(j-k)\gaq \ \mbox{ for all }\,j\ge k.\\
\bullet \  &\nu\left(A_{k,j}-\comb{j}{k}a_j h^{j-k}\right)>\al+\be_{Q,j}+(j-k)\gaq \ \mbox{ for all }\,j\ge k.\\ 
\bullet \ &\nu\left(A_{k,j}\right)>\al+\bel+(\ell-k)\gaq>\nu\left(A_{k,\ell}\right)\ \mbox{ for all }\,j>\ell.
%\end{itemize}
\end{array}
\end{equation}

Let us now argue by induction on $\ell-k>0$. Suppose $\ell-k=1$. By (\ref{always}), we have $\nu\left(A_{k,\ell}\right)<\nu\left(A_{k,j}\right)$ for all $j>\ell$; hence, in order to analyze the $\nu$-value of \,$b_k=\sum_{k\le j}A_{k,j}$, we must compare $\nu(a_k)=\nu\left(A_{k,k}\right)$ with  $\nu\left(A_{k,\ell}\right)$. We distinguish three situations.\e

(a) \ $\nu(a_k)<\nu\left(A_{k,\ell}\right)$ \ for some $Q\ge Q_2$,

(b) \ $\nu(a_k)>\nu\left(A_{k,\ell}\right)$ \ for some $Q\ge Q_2$.

(c) \ $\nu(a_k)=\nu\left(A_{k,\ell}\right)$ \ for all $Q\ge Q_2$.\e

In case (a), $\nu(b_k)=\nu(a_k)$  for all $R>Q$. Hence, the values $\be_{R,k}$ are ultimately constant and
 $k\in \bct$ with $\be_k=\nu(a_k)$.

In case (b), $\nu(b_k)=\nu\left(A_{k,\ell}\right)$ for all $R>Q$. Hence, the values $\be_{R,k}$ are ultimately constant and 
 $k\in \bct$ with $\be_k=\nu\left(A_{k,\ell}\right)$.

Let us now discuss case (c). The equality in (c), applied to $R$, shows that 
\[
\be_{R,k}=\nu(b_k)=\alkl+\bel+(\ell-k)\gar\quad\mbox{ for all }\,R>Q.                                                                         
\]
Let us show that $k\in \bunb$. Let us fix some $Q\ge Q_2$. By Lemma \ref{allsharp}, $\be_{Q,k}$ is sharp in $\bb_k$. Hence, for every $\be\in\bb_k$ there exists $h\in H$ such that $\be=\be_{Q,k}+h$. Take $Q'>Q$ such that $(\ell-k)(\ga_{Q'}-\gaq)>h$. Then,  for all $R\ge Q'$ we have
\[
\be <\be_{Q,k}+(\ell-k)(\gar-\gaq)=\alkl+\bel+(\ell-k)\gar=\be_{R,k}.
\]
This proves that $k\in \bunb$. In particular, for any $Q\ge Q_2$ we have
\[
\nu(b_k)=\be_{R,k}>\al+\bel+(\ell-k)\gaq\quad\mbox{  for all \,$R$ large enough}. 
\]

Moreover,  $n_k=n_{\ell-1}=\comb{\ell}{k}$ and the inequalities in (\ref{always}) show that
\[
 \nu(b_k-a_k-n_ka_\ell h)>\al+\bel+(\ell-k)\gaq.
\]
This proves (\ref{claim}) in the case $\ell-k=1$.

Now, suppose $k<\ell-1$ and (\ref{claim}) holds for all $j$, $k<j<\ell$.
For any pair of polynomials $f,g\in\kx$ let us denote
\[
 f\sim g\quad \sii\quad \nu(f-g)>\al+\bel+(\ell-k)\gaq.
\]
The inequalities in (\ref{always}) show that
\begin{align*}
 b_k\sim A_{k,k}+\cdots +A_{k,\ell}&\  \sim\ \left[a_k\right]_0+\cdots+ \comb{j}{k}\left[a_jh^{j-k}\right]_0+\cdots+\comb{\ell}{k}\left[a_\ell h^{\ell-k}\right]_0\\&\ \sim\ a_k+\cdots+ \comb{j}{k}a_jh^{j-k}+\cdots+\comb{\ell}{k}a_\ell h^{\ell-k}.
\end{align*}

Now, for $k<j<\ell$, the inequality (\ref{claim}) shows that $ a_j\sim -n_ja_\ell h^{\ell-j}$. Hence
\[
\comb{j}{k}a_jh^{j-k}\sim -\comb{j}{k}n_ja_\ell h^{\ell-k}.
\]
We deduce that
\begin{equation}\label{bksim}
b_k\sim  a_k-\cdots- \comb{j}{k}n_ja_\ell h^{\ell-k}-\cdots+\comb{\ell}{k}a_\ell h^{\ell-k}=a_k+n_ka_\ell h^{\ell-k}. 
\end{equation}
As in the case $k=\ell-1$, we distinguish three different cases:
\e

(a) \ $\nu(a_k)<\nu\left(n_ka_\ell h^{\ell-k}\right)$ \ for some $Q\ge Q_2$, leading to $k\in \bct$ with $\be_k=\nu(a_k)$.

(b) \ $\nu(a_k)>\nu\left(n_ka_\ell h^{\ell-k}\right)$ \ for some $Q\ge Q_2$, leading to $k\in \bct$ with $\be_k=\nu(n_k)+\bel+(\ell-k)\gaq$.

(c) \ $\nu(a_k)=\nu\left(n_ka_\ell h^{\ell-k}\right)$ \ for all $Q\ge Q_2$.\e

The equality in (c), applied to $R$, shows that 
\[
\be_{R,k}=\nu(b_k)=\nu(n_k)+\bel+(\ell-k)\gar\quad\mbox{ for all }\,R>Q.                                                                         
\]
The arguments above  show that $k\in \bunb$.  In particular, for any $Q\ge Q_2$ we have
\[
\nu(b_k)>\al+\bel+(\ell-k)\gaq\quad\mbox{  for all \,$R$ large enough}. 
\]
Hence, $b_k\sim 0$ and the inequality in (\ref{claim}) follows from (\ref{bksim}).
\end{proof}%\e

\begin{corollary}\label{pinH}
Suppose that $\qq$ is VU, $\si=1$, $v(p)\in H$ and $\bunb=\{0\}$. Then, $d=1$ and $\bct=( 0,D]\cap\Z$. 
\end{corollary}

\begin{proof}
The equality  $\bct=(0,D]\cap\Z$ follows immediately from Theorems \ref{B1=B0} and \ref{Jempty}. The equality $d=1$ follows then from Theorem
 \ref{minB1}.   
\end{proof}

\subsection{The unbounded case}
This is the VU case in which $H=\g$.

For an arbitrary polynomial $g\in\kx$, let $\mlt(g)$ be the least positive integer $s$ such that $\ps g\ne0$. That is, 
$\mlt(g)$ is the largest  exponent $s$ such that $g$ is a polynomial in $x^{p^s}$. 

\begin{theorem}\cite[Theorem 4.11]{AFFGNR}\label{U}
Let $\qq$ be unbounded. Then, $\kpi(\qq)=\{F\}$ is a one-element set, and $\mlt(F)=ds_\infty$.
\end{theorem}

The fact that $\kpi(\qq)$ is a one-element set follows immediately from Lemma \ref{MinAug}. The following result is an immediate consequence of $\mlt(F)=1$.

\begin{corollary}\label{separable}
If $F$ is a separable polynomial, then $d=\si=1$.	
\end{corollary}

\subsection{The degree-one case}\label{subsecdeg1}
In this section, we assume that $m=\deg\qq=1$.

\begin{lemma}\label{proposdegrefixed}
If $\deg\qq=1$, then $\{\be_{Q,\ell}\}_{Q\in\mathcal Q}$ is ultimately constant for all $\ell>0$.
\end{lemma}
\begin{proof}
For each $Q\in \mathcal Q$ write $Q=x-a_Q$. Then $\underline{a}=\{a_Q\}_{Q\in \mathcal Q}$ is a pseudo-convergent sequence in $K$. By \cite[Corollary 3.4]{NS2018} every polynomial in $\kx_{m_\infty}$ is fixed by $\underline a$. In particular, the value of $\partial_\ell F$ is fixed by $\underline a$, for every $\ell>0$.
Since
\[
F=F(a_Q)+\partial F(a_Q)Q+\ldots+\partial_DF(a_Q)Q^D
\]
we deduce that $\beta_{Q,\ell}=\nu (\partial_\ell F(a_Q))$. Hence, $\{\beta_{Q,l}\}_{Q\in \mathcal Q}$ is ultimately constant.
\end{proof}\e

%The next result follows immediately from the above lemma and the definition.

\begin{corollary}
If $\deg(\mathcal Q)=1$, then $B=\{0\}\cup\bct$.
\end{corollary}
 
In the degree-one case,  we have $\ep_Q=\gaq$ and $I(Q)=\{1\}$ for all $Q\in\qq$. Hence,  $\si=1$ and the following result is an immediate consequence of Corollary \ref{pinH}.

\begin{theorem}\label{vpsmall}
 Suppose that $\qq$ is VU, has degree one and $v(p)\in H$. Then, $d=1$ and $\bct=[1,D]\cap\N$.
\end{theorem}

In the case $v(p)>H$, we are able to obtain more information on the structure of the set $B$ because the comparison of $Q$-expansions of $F$ is much more simple when  $m=1$. Suppose that $R>Q$ and let the corresponding expansions of $F$ be: 
\[
F=a_0+a_1Q+\cdots+a_DQ^D=b_0+b_1R+\cdots+b_DR^D.
\]
Take $Q$ large enough to stabilize the values $\bel=\be_{Q,\ell}$ for all $1\le\ell\le D$.
Write $Q=R+h$ for some $h\in K$. Then,
\[
 b_k=\sum\nolimits_{k\le \ell}A_{k,\ell},\qquad A_{k,\ell}=\comb{\ell}{k}a_\ell h^{\ell-k}.
\]
With the notation of (\ref{alkj}), we have  $\nu\left(A_{k,\ell}\right)=\alkl+\bel+(\ell-k)\gaq$.

Let us recall a classical criterion to decide when $p$ divides a binomial number.

\begin{lemma}
 For arbitrary $\ell,k\in\N$, consider their $p$-adic expansions:
 \[
\ell=r_0+r_1p+\cdots+r_np^n,\qquad k=s_0+s_1p+\cdots+s_np^n,
 \]
completing with zero coefficients the shorter one, if necessary. Then,
\[
p\mid \comb{\ell}{k}\ \sii\ \mbox{ there exists $0\le i\le n$ such that }\ r_i<s_i. 
\]
\end{lemma}

\begin{lemma}\label{previous}
If $v(p)>H$, then $p\mid \comb{\ell}{k}$, for all $k\in J$ and  all $\ell\in B$, $\ell>k$. 
\end{lemma}

\begin{proof}
 Consider $Q$ large enough to stabilize the values $\bel=\be_{Q,\ell}$ for all $\ell>0$. Take $k\in J$ and suppose there exists some $\ell\in B$, $\ell>k$, such that $p\nmid \comb{\ell}{k}$. Take $\ell$ minimal satisfying these properties. Since $\alkl=0$, we have $\nu\left(A_{k,\ell}\right)=\bel+(\ell-k)\gaq$ and we claim that
\begin{equation}\label{claim7}
\nu(b_k)=\nu\left(A_{k,\ell}\right)\quad\mbox{ for all }\,R>Q.                                                                                                                                                                \end{equation}
Since $\bel+\ell\gaq\in\dta^L$ (because $\ell\in B$), and $\gar-\gaq\in H$, this implies 
\[
\nu\left(b_kR^k\right)=\bel+(\ell-k)\gaq +k\gar=\bel+\ell\gaq+k(\gar-\gaq)\in\dta^L.                                                                                                                                                                                                                                                                                                                                                                                                                                                                     \]
Since this holds for all $R>Q$, it contradicts the assumption $k\in J$. 

In order to prove (\ref{claim7}), it suffices to show 
that 
\begin{equation}\label{middle}
\be_j+j\gaq>\bel+\ell\gaq, 
\end{equation}
for all $k\le j\le D$, $j\ne\ell$.
Indeed, this inequality  will imply
\[
\nu\left(A_{k,j}\right)\ge\be_j+(j-k)\gaq>\bel+(\ell-k)\gaq=\nu\left(A_{k,\ell}\right).
\]

If $j\in J$, then $\be_j+j\gaq>\dta^L$. Hence, (\ref{middle}) follows from the fact that  $\bel+\ell\gaq\in\dta^L$.

If $j\in B$ and $j>\ell$, then (\ref{kaplansky}) and Corollary \ref{sg=1} imply (\ref{middle}). 

Finally, suppose $j\in B$, $j<\ell$. Then, $p\mid \comb{j}{k}$ and $\al_{k,j}>H$. By Lemmas \ref{wholecut} and \ref{allsharp},
$\be_j+j\gaq+H$ is a final segment of $\dta^L$.
 Hence, $\al_{k,j}+\be_j+j\gaq>\dta^L$, so that
\[
\nu\left(A_{k,j}\right)=\al_{k,j}+\be_j+(j-k)\gaq>\bel+(\ell-k)\gaq=\nu\left(A_{k,\ell}\right)
\]

This ends the proof of (\ref{claim7}). 
\end{proof}\e

\begin{lemma}\label{mainshort}
Suppose that $\qq$  is VU, has $\deg(\qq)=1$ and $v(p)>H$. 
 If $p^re\in B$, with $p\nmid e$, then $p^r\in B$.
\end{lemma}

%\noindent{\bf Proof of Theorem \ref{mainm=1} when $v(p)>H$.}\e

\begin{proof}
If $e=1$, we have $p^r\in B$ already. If $e>1$, then $p^r<\ell$ and $p\nmid\comb{\ell}{p^r}$.
Hence, $p^r$ belongs to $B$  by Lemma \ref{previous}. 
\end{proof}\e

We obtain the following relevant consequence.

\begin{theorem}\label{mainm=1}
Suppose that $\qq$ is VU, has degree one and $v(p)>H$. Then,  $\ord_p(\ell)\ge\ord_p(d)$
for all $\ell\in B$. 
\end{theorem}

\begin{proof}
Take any $\ell\in B$ and let $r=\ord_p(\ell)$. By Lemma \ref{mainshort}, $p^r$ belongs to $B$.

The condition $r<\ord_p(d)$ implies $p^r<d$ and this contradicts 
Theorem
 \ref{minB1}.     
\end{proof}\e

%On behalf of all authors, the corresponding author states that there is no conflict of interest.

\end{document}